\documentclass[a4paper]{amsart}
\usepackage[utf8]{inputenc}
\usepackage{amssymb,amsmath,graphics,pstricks,tikz}
\usepackage{a4wide}
\usepackage[figuresright]{rotating}
\usepackage[english]{babel}
\usepackage{amsthm}
\usepackage{bbold,latexsym}
\usepackage{color}
\usepackage{amsfonts}
\usepackage{pstricks}
\usepackage{fourier}
\usepackage{paralist}
\usepackage{graphicx,multirow}
\usepackage{subfigure}
\usepackage{comment}
\usepackage[figuresright]{rotating}

\usepackage[pdftex,
            pdftitle={Fully discrete approximation of the semilinear stochastic wave equation on the sphere},
            pdfauthor={D. Cohen, S. Di Giovacchino, and A. Lang},
            bookmarksopen,
            colorlinks,
            linkcolor=black,
            urlcolor=black,
            citecolor=black
]{hyperref}

\newtheorem{thm}{Theorem}
\newtheorem{prop}[thm]{Proposition}
\newtheorem{cor}[thm]{Corollary}
\newtheorem{lem}[thm]{Lemma} 

\theoremstyle{remark}

\theoremstyle{definition}
\newtheorem{ass}[thm]{Assumption}

\newcommand{\R}{{\mathbb R}}
\newcommand{\N}{\mathbb{N}}
\newcommand{\HH}{\mathbb{H}}
\newcommand{\E}{\mathbb{E}}

\newcommand{\IS}{\mathbb{S}}
\newcommand*\dd{\mathop{}\!\mathrm{d}}
\newcommand{\e}{\mathrm{e}}
\newcommand{\ii}{\ensuremath{\mathrm{i}}}
\newcommand{\IT}{\mathbb{T}}

\newcommand{\gb}{\beta}

\newcommand{\vp}{\varphi}
\newcommand{\vt}{\vartheta}

\newcommand{\KL}{Karhunen--Lo\`eve }
\newcommand{\trace}{\operatorname{Tr}}
\newcommand{\floor}[1]{\lfloor #1 \rfloor}

\definecolor{darkgreen}{rgb}{0,.6,0}

\begin{document}
\title{Fully discrete approximation of the semilinear stochastic wave equation on the sphere}
\date{\today}

\author[D.~Cohen]{David Cohen}
\author[S.~Di Giovacchino]{Stefano Di Giovacchino}
\author[A.~Lang]{Annika Lang}

\address[David Cohen]{Department of Mathematical Sciences, 
              Chalmers University of Technology \& University of Gothenburg, 
              S--412 96 Gothenburg, Sweden.} 
              \email[]{david.cohen@chalmers.se}
\address[Stefano Di Giovacchino]{Department of Information Engineering and Computer Science and Mathematics, University of L’Aquila, I--67 100 L’Aquila, Italy.}
          \email[]{stefano.digiovacchino@univaq.it}
\address[Annika Lang]{Department of Mathematical Sciences, 
              Chalmers University of Technology \& University of Gothenburg, 
              S--412 96 Gothenburg, Sweden.}
              \email[]{annika.lang@chalmers.se}
            

\begin{abstract}
    The semilinear stochastic wave equation on the sphere driven by multiplicative Gaussian noise is discretized by a stochastic trigonometric integrator in time and a spectral Galerkin approximation in space based on the spherical harmonic functions. Strong and almost sure convergence of the explicit fully discrete numerical scheme are shown. Furthermore, these rates are confirmed by numerical experiments.
\end{abstract}

\maketitle

\section{Introduction}

Today, it is well accepted that stochastic partial differential equations (SPDEs) provide essential and relevant models to describe phenomena in finance~\cite{MR2732847}, physics~\cite{PhysRevLett.56.889}, biology~\cite{Dalang2009}, or neurology~\cite{1002247}, to name but a few examples of possible applications. The numerical analysis of SPDEs on Euclidean spaces is a mature, vibrant, and active area of research since the seminal paper~\cite{MR1341554}. Relevant to the present publication, we mention the following papers on the numerical analysis of stochastic wave equations defined on Euclidean domains: \cite{MR2318291,MR2379913,MR2646102,MR3033008,MR3276429,MR3463447,MR3484400,MR3573128,MR4285774,MR4403208,MR4362830,MR4403208,MR4727112,MR4744301,MR4846352,MR4896748, MR4982673}.

Moving beyond the well-understood Euclidean setting, results on the numerical analysis of time-dependent SPDEs defined on surfaces, and especially on the unit sphere, have only recently emerged in the literature. To our knowledge, no approximations of semilinear equations have been considered so far and only \cite{MR3907363} treats multiplicative noise, while all other results are limited to the additive case. Without being exhaustive, we refer the interested reader for results on spectral approximations to: 
the works \cite{MR3404631, MR4701212, lang2025approximationlevydrivenstochasticheat}, which consider linear stochastic heat equations with an exact and Euler--Maruyama approximation in time and show different types of strong and weak convergence for Wiener and L\'evy noise;
the paper~\cite{MR4462619}, where similar results are derived for the linear stochastic wave and the stochastic Schr\"odinger equations;
the numerical approximation of linear fractional heat equations in~\cite{123456}, of a linear time-fractional diffusion equation in~\cite{MR4714764}, and of a linear stochastic Stokes equation in~\cite{LeGia2019};
and the only available convergence analysis for an equation on the sphere driven by multiplicative noise, more specifically the stochastic heat equation in~\cite{MR3907363}.

Building upon recent developments in the Euclidean setting (e.g. \cite{MR3353942,MR3484400}) and in the Riemannian one (e.g. \cite{MR4462619}), the present paper offers the first numerical analysis of a fully discrete numerical method for a semilinear SPDE on the sphere driven by multiplicative noise. In particular, we prove strong rates of convergence of a spectral approximation combined with an explicit stochastic trigonometric integrator for the semilinear stochastic wave equation with multiplicative noise on the sphere
\begin{equation*}
\partial_{tt}{u}(t)=\Delta_{\mathbb{S}^2}u(t)+f(u(t))+g(u(t))\dot{W}(t),
\end{equation*}
where $f,g$ are given globally Lipschitz functions, $\Delta_{\IS^2}$ is the spherical Laplacian and $W$ is an isotropic $Q$-Wiener process defined over the unit sphere~$\IS^2$ (precise definitions and assumptions are given in Section~\ref{sect-fram}). The main contributions of this work are Theorem~\ref{spa_thm} on the strong rate of convergence of a spectral method for the above SPDE and Theorem~\ref{tem_thm} on the strong rate of convergence of an explicit time integrator for this stochastic wave equation on the sphere. 

The outline of the paper is as follows: In Section~\ref{sect-fram}, we give the setting and present the Wiener process on the sphere. This is then used to formulate the considered stochastic wave equation on the sphere as an abstract stochastic evolution equation. Finally, well-posedness and regularity of the exact solution to this abstract evolution problem are given. In Section~\ref{sec-space}, we analyze the strong convergence of a spectral approximation by the spherical harmonics for the semilinear stochastic wave equation on the sphere. Section~\ref{sec-time} is devoted to the proof of strong convergence of the stochastic trigonometric integrator for the time discretization of the considered SPDE. Numerical experiments illustrating our theoretical findings are presented in Section~\ref{sec-num}.

\section{Framework}\label{sect-fram}
In this section, we introduce notation, the Wiener process on the sphere and 
define the considered semilinear stochastic wave equation. 
To do this, we mainly follow the work \cite{MR4462619} on the linear case. This preparation is then used to show existence and smoothness of solutions to the considered SPDE. 

\subsection{Setting and Wiener process on the sphere} 
We denote a complete filtered probability space by $(\Omega, \mathcal{F}, (\mathcal{F}_t)_{t\in[0,T]},\mathbb{P})$ satisfying the usual conditions. We consider the unit sphere $\mathbb S^{2}=\{x\in\R^3, \|x\|_{\R^3}=1\}$, where $\|\cdot\|_{\R^3}=(\langle\cdot,\cdot\rangle_{\R^3})^{1/2}$ denotes the Euclidean norm in $\R^3$. On the unit sphere $\IS^2$, we consider the geodesic metric given by $d(x,y)=\arccos(\langle x,y\rangle_{\R^3})$ for any $x,y\in\IS^2$. Furthermore, we identify Cartesian coordinates $x\in\IS^2$ with spherical coordinates $(\vartheta,\varphi)\in[0,\pi]\times[0,2\pi)$ by the transformation 
$x=(\sin(\vartheta)\cos(\varphi),\sin(\vartheta)\sin(\varphi),\cos(\vartheta))$. 

The generalization of the Laplace operator to functions defined on the sphere is called the spherical Laplacian or the Laplace--Beltrami operator. In terms of spherical coordinates, this operator 
is defined as
$$
\Delta_{\IS^2}=\sin(\vartheta)^{-1}\frac{\partial}{ \partial_\vartheta}\left( \sin(\vartheta) \frac{\partial}{ \partial_\vartheta}\right)+\sin(\vartheta)^{-2}\frac{\partial^2}{ \partial_\varphi^2},
$$
see \cite[Section~3.4.3]{MP11} for instance. 
The eigenvalues of the Laplace--Beltrami operator are for $\ell \in \N_0$
$$
\Delta_{\IS^2} Y_{\ell,m} = -\ell(\ell+1) Y_{\ell,m}
$$
with the eigenfunctions given by the spherical harmonics $\mathcal{Y}= (Y_{\ell,m})_{\ell\in\mathbb N_0, m\in\{-\ell,\ldots,\ell\}}$, where $Y_{\ell,m}\colon[0,\pi]\times[0,2\pi) \to \mathbb C$ is 
\begin{equation*}
Y_{\ell, m}(\vartheta, \varphi)
= \sqrt{\frac{2\ell + 1}{4\pi}\frac{(\ell-m)!}{(\ell+m)!}} P_{\ell, m}(\cos(\vartheta)) \e^{\ii m\varphi}
\end{equation*}
for $\ell \in \mathbb N_0$, $m = 0,\ldots, \ell$, and $(\vartheta,\varphi) \in [0,\pi] \times [0,2\pi)$, and
\begin{equation*}
Y_{\ell, m}
= (-1)^m \overline{Y_{\ell, -m}}
\end{equation*}
for $\ell \in \mathbb N$ and $m=-\ell, \ldots,-1$. Here, $P_{\ell, m}$ are the associated Legendre polynomials, see \cite[Theorem~2.13]{M98} for details. 

Let $\sigma$ denote the Lebesgue measure on the sphere. 
The subspace $L^2(\IS^2)=L^2(\IS^2;\mathbb R)$ of real-valued functions in $L^2(\IS^2;\mathbb C)$ is a Hilbert space with the inner product defined by
\[
\langle f,g \rangle_{L^2(\IS^2)}=\int_{\IS^2} \langle f(y),g(y) \rangle_{\mathbb R^3} \dd \sigma(y)
\]
for $f,g\in L^2(\IS^2)$. It is well-known that the spherical harmonics form an orthonormal basis of the space $L^2(\IS^2)$, see e.g. \cite[Chapter 7]{Kosmann-Schwarzbach2010}. 

It is now time to define the driving noise of the considered stochastic wave equation on the sphere. 
We refer to \cite{MR3404631} for details. Let $T\in(0,\infty)$ and define the time interval $\IT=[0,T]$. 
In the work \cite{MR3404631}, it is shown that an isotropic $Q$-Wiener process $(W(t), t \in \IT)$, taking values in $L^2(\IS^2)$, can be expressed by the series expansion 
\begin{align}\label{eqW}
\begin{split}
W(t,y)& = \sum_{\ell=0}^\infty \sum_{m=-\ell}^\ell a^{\ell, m}(t) Y_{\ell, m}(y)
\\
&= \sum_{\ell=0}^\infty \left(\sqrt{A_\ell} \gb_1^{\ell, 0}(t) Y_{\ell, 0}(y) 
+ \sqrt{2A_\ell} \sum_{m=1}^\ell (\gb_1^{\ell, m}(t) \text{Re} Y_{\ell, m}(y) + \gb_2^{\ell, m}(t) \text{Im} Y_{\ell, m}(y))\right)
\\
& = \sum_{\ell=0}^\infty \left( \sqrt{A_\ell} \gb_1^{\ell, 0}(t) L_{\ell, 0}(\vt) 
+ \sqrt{2A_\ell} \sum_{m=1}^\ell L_{\ell, m}(\vt) (\gb_1^{\ell, m}(t) \cos(m\vp) + \gb_2^{\ell, m}(t) \sin(m \vp))\right),
\end{split}
\end{align}
where $(\gb_i^{\ell, m}, i=1,2, \ell \in \N_0, m=0,\ldots, \ell)$ 
is a sequence of independent, real-valued Brownian motions, with 
$\gb_2^{\ell, 0} = 0$ for $\ell \in \N_0$ and $t \in \IT$. In addition, for $\ell \in \N$, $m = 1,\ldots,\ell$, and $\vt \in [0,\pi]$ we have set
  \begin{equation*}
   L_{\ell, m}(\vt)
    = \sqrt{\frac{2\ell +1}{4\pi} \frac{(\ell-m)!}{(\ell+m)!}}
	P_{\ell, m}(\cos(\vt)).
  \end{equation*}
The covariance operator $Q$ is similarly characterized as in the introduction of \cite{LLS13} by 
\begin{align*}
Q Y_{\ell, m} 
& = A_\ell Y_{\ell, m}
\end{align*}
for $\ell \in \N_0$ and $m=-\ell,\ldots,\ell$, meaning that the eigenvalues of $Q$ 
are represented by the given angular power spectrum $(A_\ell, \ell \in \N_0)$, 
and the eigenfunctions are the spherical harmonic functions. 

\subsection{The semilinear stochastic wave equation on the sphere}
We now have the necessary material in place to define the stochastic wave equation on the sphere, which we will numerically discretize in the next sections.

We consider the semilinear stochastic wave equation on the sphere $\IS^{2}$ 
driven by an isotropic $Q$-Wiener process
\begin{equation}
\label{stoc_wave}
\partial_{tt}{u}(t)=\Delta_{\mathbb{S}^2}u(t)+f(u(t))+g(u(t))\dot{W}(t),\qquad 
\text{for $t\in(0,T]$}
\end{equation}
with initial values $u(0)=u_0$ and $\partial_t{u}(0)=v_0$. The notation $\dot W$ denotes the formal derivative of the $Q$-Wiener process $W$ given by equation~\eqref{eqW}. Here, the given mappings $f,g(u)\colon L^2(\IS^2) \to L^2(\IS^2)$, for $u\in L^2(\IS^2)$, are deterministic and verify some smoothness assumptions given below. 

We start by writing the stochastic wave equation~\eqref{stoc_wave} as the system of 
stochastic evolution equations 
\begin{equation}
\label{abstract_eq}
\dd X(t)=A X(t)\dd t+F(X(t))\dd t+G(X(t))\dd W(t), \qquad X(0)=X_0,
\end{equation}
where
$$
X(t)=\begin{pmatrix}u(t)\\v(t)\end{pmatrix}=\begin{pmatrix}u(t)\\\partial_t u(t)\end{pmatrix}, \:
A=\begin{pmatrix}
0& I\\ \Delta_{\mathbb{S}^2}&0\end{pmatrix}, \: F(X(t))=\begin{pmatrix}0\\
f(u(t))\end{pmatrix}, \: G(X(t))=\begin{pmatrix}0\\ g(u(t))\end{pmatrix}, \text{ and } \: 
X_0=\begin{pmatrix}u_0\\ v_0\end{pmatrix}.
$$
To show existence and uniqueness of solutions to~\eqref{abstract_eq}, we define the Sobolev spaces $H^s=H^s(\IS^2)=\left(I-\Delta_{\IS^2}\right)^{-s/2}L^2(\IS^2)$ for any $s\in\R$, and their product spaces $\mathbb{H}^s=\mathbb{H}^s(\mathbb{S}^2) = H^s\times H^{s-1}$  
with the norms
\begin{equation}
\label{norm_def}
\left\|X\right\|^2_{{\HH}^s}=\left\|u\right\|_{s}^2+\left\|v\right\|_{{s-1}}^2
\end{equation}
for $X=(u,v)^T\in \HH^s$. Here, we have used the norms  
$$
\|u\|_s=\| \left(I-\Delta_{\IS^2}\right)^{s/2} u\|_{L^2(\IS^2)}
$$
for $u\in H^s$. For $s=0$, we set $H^0=L^2(\IS^2)$ and $\HH=\HH^0=H^0\times H^{-1}$. Note that $\HH$ is a separable Hilbert space.

Finally, we denote by $\left\|\cdot\right\|_{\mathcal{L}(\HH^s)}$ the usual operator norm in the space $\mathcal{L}(\HH^s)=\mathcal{L}(\HH^s;\HH^s)$ of bounded linear operators from $\HH^s$ to $\HH^s$. Similarly, the space $\mathcal{L}_2(H^{s})$ denotes the space of the Hilbert--Schmidt operator from $L^2(\mathbb{S}^2)$ to $H^{s}(\mathbb{S}^2)$. 

The mild solution to the stochastic evolution equation~\eqref{abstract_eq}, as for instance defined in ~\cite{MR3236753}, reads 
\begin{equation}
\label{mild_form_comp}
X(t)=E(t)X_0+\int_{0}^{t}E(t-s)F(X(s))\dd s+\int_{0}^{t}E(t-s) G(X(s))\dd W(s)
\end{equation}
for all $t\in \mathbb{T}$. Here, the $C_0$-group $E(t)=e^{tA}$ is given by 
$$
E(t)=
\begin{pmatrix}C(t) & \left(-\Delta_{\mathbb{S}^2}\right)^{-\frac{1}{2}}S(t)\\
-\left(-\Delta_{\mathbb{S}^2}\right)^{\frac{1}{2}}S(t) & C(t)
\end{pmatrix},
$$ 
where we set, for $t \in \R$, $C(t)=\cos\left(t\left(-\Delta_{\mathbb{S}^2}\right)^{\frac{1}{2}}\right)$ and 
$S(t)=\sin\left(t\left(-\Delta_{\mathbb{S}^2}\right)^{\frac{1}{2}}\right)$ for the cosine and sine operators. 

We now collect some results on the above cosine and sine operators on the sphere. 
These results are very similar to the corresponding results from the Euclidean setting given in \cite{MR2646102,MR3033008,MR3353942} for instance. However, since one eigenvalue of the Laplace--Beltrami operator is zero, the proofs of these results on the sphere need to be adapted. First, using the expansions of the cosine and sine operators in spherical harmonics together with the trigonometric identity $\cos(t)^2+\sin(t)^2=1$, one can get the following estimates.
\begin{lem}
\label{lem_op}
For any $s,t\in\mathbb R$, $u\in H^s(\mathbb{S}^2)$ and $v \in H^{s-1}(\mathbb{S}^2)$, the elements of the group matrix $E(t)$ are bounded by
$$
\begin{aligned}
\left\|C(t)u \right\|_{s} & \le \left\|u\right\|_{s}, \quad 
\left\|\left(-\Delta_{\mathbb{S}^2}\right)^{-\frac{1}{2}}S(t) v \right\|_{s} \le 
\max\left(|t|,\sqrt{\frac32}\right)\left\|v\right\|_{{s-1}}, \quad
\left\|\left(-\Delta_{\mathbb{S}^2}\right)^{\frac{1}{2}} S(t)  u \right\|_{{s-1}} \le \left\|u\right\|_{{s}}.
\end{aligned}
$$
These bounds yield the group estimate 
\begin{equation*} 
\left\|E(t)\right\|_{\mathcal{L}(\HH^s)}\leq 2\max\left(|t|,\sqrt{\frac32}\right).
\end{equation*}
\end{lem}
\begin{proof}
Let us denote by $\{u^{\ell,m}\}_{\ell=0,m=-\ell}^{\infty,\ell}$ the coefficients of $u$ in the basis given by the spherical harmonics. By definition of the norms and orthonormality of the spherical harmonics, we have the estimates
$$
\begin{aligned}
 \left\|C(t)u \right\|^2_{s} &=  \left\|\left(I-\Delta_{\mathbb{S}^2}\right)^{\frac{s}{2}}\cos(t(-\Delta_{\mathbb{S}^2})^{\frac{1}{2}})u \right\|^2_{0}
 =\sum_{\ell=0}^{\infty}\sum_{m=-\ell}^{\ell}(1+\ell(\ell+1))^s \cos^2(t (\ell(\ell+1))^{\frac{1}{2}})|u^{\ell,m}|^2\\
 & \le \sum_{\ell=0}^{\infty}\sum_{m=-\ell}^{\ell}(1+\ell(\ell+1))^s |u^{\ell,m}|^2 
 = \left\|u\right\|^2_{s}.
\end{aligned}
$$
This shows the first estimate in the lemma. 

Let us now denote by $\{v^{\ell,m}\}_{\ell=0,m=-\ell}^{\infty,\ell}$ 
the coefficients of $v$ in the basis given by the spherical harmonics. As above, by definition of the norms, we obtain analogously
$$
\begin{aligned}
\left\|\left(-\Delta_{\mathbb{S}^2}\right)^{-\frac{1}{2}}S(t) v \right\|^2_{s} 
&=\left\|(I-\Delta_{\mathbb{S}^2})^{\frac{s}{2}} (-\Delta_{\mathbb{S}^2})^{-\frac{1}{2}}\sin(t(-\Delta_{\mathbb{S}^2})^{\frac12}) v\right\|^2_{0}=\left\|(I-\Delta_{\mathbb{S}^2})^{\frac{s}{2}} t\text{sinc}(t(-\Delta_{\mathbb{S}^2})^{\frac12}) v\right\|^2_{0}\\
&=t^2|v^{0,0}|^2+\sum_{\ell=1}^{\infty}\sum_{m=-\ell}^{\ell}(1+\ell(\ell+1))^s 
(\ell(\ell+1))^{-1}\sin^2(t(\ell(\ell+1))^{\frac{1}{2}})|v^{\ell,m}|^2\\
&\leq t^2|v^{0,0}|^2+\frac32\sum_{\ell=1}^{\infty}\sum_{m=-\ell}^{\ell}(1+\ell(\ell+1))^{s-1}|v^{\ell,m}|^2\\
& \le \max\left(t^2,\frac32\right) \sum_{\ell=0}^{\infty}\sum_{m=-\ell}^{\ell}(1+\ell(\ell+1))^{s-1} |v^{\ell,m}|^2
= \max\left(t^2,\frac32\right) \left\|v\right\|^2_{s-1}.
\end{aligned}
$$
This proves the second estimate of the lemma. The proof of the last estimate of the lemma is done similarly and is thus omitted. Combining the above bounds for the cosine and sine operators yields the bound for $E$ and concludes the proof.
\end{proof}

We collect further estimates on the cosine and sine operators in the next lemma.
\begin{lem}
\label{lem_op2}
For any $s\in\mathbb{R}$, $q \in [0,1]$, $t \in \IT$, $u \in H^{s+q}(\IS^2)$, and $v \in H^{s-1+q}(\IS^2)$, the sine and cosine operators satisfy the bounds
\begin{gather*}
    \| (C(t) - I)u \|_s
         \le 2 t^q \|u\|_{s+q}, \quad
    \| (-\Delta_{\mathbb{S}^2})^{-\frac{1}{2}} S(t)v \|_s
        \le \max\left(T, \sqrt{\tfrac{3}{2}}\right) t^q \|v\|_{s-1+q}, \\
    \left\|(-\Delta_{\mathbb{S}^2})^{\frac{1}{2}}S(t) u\right\|_{s-1}
         \le t^q \left\|u\right\|_{s+q}.
\end{gather*}
These estimates bound the group~$E$ by
\begin{equation*} 
    \| E(t) - I \|_{\mathcal{L}(\HH^{s+q};\HH^s)}
        \le \max (4, 2T) t^q .
\end{equation*}
\end{lem}

\begin{proof}
    With the same expansion as in the proof of Lemma~\ref{lem_op}, we obtain
    \begin{align*}
        \| (C(t) - I)u \|_s^2
            = \sum_{\ell=0}^\infty \sum_{m=-\ell}^\ell (1+\ell(\ell+1))^{s} (\cos(t(\ell(\ell+1))^{1/2}) - 1)^2 |u^{\ell,m}|^2.
    \end{align*}
    We observe that the summand for $\ell=0$ is zero and that for $\ell \neq 0$ on the one hand 
    \begin{equation*}
        (\cos(t(\ell(\ell+1))^{1/2}) - 1)^2 \le 4
    \end{equation*}
    and on the other hand 
    \begin{equation*}
        (\cos(t(\ell(\ell+1))^{1/2}) - 1)^2
            = \left( \int_0^{t(\ell(\ell+1))^{1/2}} \sin(r) \, \dd r \right)^2
            \le t^2 \ell(\ell+1)
            \le t^2 (1+\ell(\ell+1)),
    \end{equation*}
    since $|\sin(\cdot)|$ is bounded by~$1$. Plugging these bounds into the series expansion therefore yields
    \begin{equation*}
        \| (C(t) - I)u \|_s^2
        \le
        \begin{cases}
            4 \|u\|_s^2 & \text{for } u \in H^s(\IS^2),\\
            t^2 \|u\|_{s+1}^2 & \text{for } u \in H^{s+1}(\IS^2).
        \end{cases}
    \end{equation*}
    The claim follows by interpolation.

    To show the second claim, we use the bound from Lemma~\ref{lem_op} and observe further, with the boundedness of the cosine function, that
    \begin{equation*}
        \sin^2(t(\ell(\ell+1))^{1/2})
            = \left( \int_0^{t(\ell(\ell+1))^{1/2}} \cos(r) \, \dd r\right)^2
            \le t^2 \ell(\ell+1),
    \end{equation*}
    which plugged into the series expansion in the proof of Lemma~\ref{lem_op} yields
    \begin{equation*}
        \| (-\Delta_{\mathbb{S}^2})^{-\frac{1}{2}} S(t)v \|_s^2
        \le
        \begin{cases}
            \max\left(t^2,\tfrac{3}{2}\right) \|v\|_{s-1}^2 & \text{for } v \in H^{s-1}(\IS^2),\\
            t^2 \|v\|_{s}^2 & \text{for } v \in H^{s}(\IS^2).
        \end{cases}
    \end{equation*}
    The claim follows by interpolation and by bounding $t$ by~$T$.
    The last claim follows analogously, and its proof is therefore omitted.
\end{proof}

In order to treat the nonlinearities in the stochastic wave equation on the sphere~\eqref{stoc_wave} or equivalently the abstract SPDE~\eqref{abstract_eq}, we make the following assumptions
\begin{ass}
\label{ass_compact}
For any $u,u_1,u_2 \in L^2(\mathbb{S}^2)$, the nonlinear mappings 
$f$ and $g$ satisfy
$$
\begin{aligned}
&\left\|f(u_1)-f(u_2)\right\|_{L^2(\mathbb{S}^2)} +\left\|\left(g(u_1)-g(u_2)\right)Q^{\frac{1}{2}}\right\|_{\mathcal{L}_2(H^{\delta-1})} \le L_1 \left\|u_1-u_2\right\|_{L^2(\mathbb{S}^2)},\\
&\left\|f(u)\right\|_{L^2(\mathbb{S}^2)}+\left\|g(u)Q^{\frac{1}{2}}\right\|_{\mathcal{L}_2(H^{\delta-1})}\le L_2 \left(1+\left\|u\right\|_{L^2(\mathbb{S}^2)}\right),
\end{aligned}
$$
for some constants $\delta, L_1, L_2\in(0,\infty)$. 
\end{ass}

The finiteness of $\|g(u)Q^{1/2}\|_{\mathcal{L}_2(H^{\delta-1})}$ gives conditions on the smoothness of the noise. For additive noise, i.\,e., $g = I$, the condition simplifies to $\|Q^{1/2}\|_{\mathcal{L}_2(H^{\delta-1})} = \|(I- \Delta_{\IS^2})^{(\delta-1)/2} Q^{1/2}\|_{\mathcal{L}_2(H^0)} = \trace ((I-\Delta_{\IS^2})^{(\delta-1)} Q) < + \infty$. For $\delta = 1$, we are thus in the setting of trace class noise.

With this at hand, we can prove the following lemma.
\begin{lem}
\label{lip_comp}
Assume that the stochastic wave equation~\eqref{stoc_wave} satisfies Assumption~\ref{ass_compact}. Then, there exist positive constants $\tilde{L}_1$ and $\tilde{L}_2$, such that, for any $X=(u, v)^T$, $X_1=(u_1, v_1)^T$, and $X_2=(u_2, v_2)^T$, it holds 
$$
\begin{aligned}
&\left\|F(X_1)-F(X_2)\right\|_\HH + \left\|\left(G(X_1)-G(X_2)\right)Q^{\frac{1}{2}}\right\|_{\mathcal{L}_2(\HH)}\le \tilde{L}_1\left\|X_1-X_2\right\|_\HH,\\
&\left\|F(X)\right\|_{\HH} + \left\|G(X)Q^{\frac{1}{2}}\right\|_{\mathcal{L}_2(\HH)}\le \tilde{L}_2\left(1+\left\|X\right\|_\HH\right),
\end{aligned}
$$
where $\mathcal{L}_2(\HH) = \mathcal{L}_2(L^2(\IS^2);\HH)$. 
\end{lem}
\begin{proof}
By definitions of the mappings $F$ and $G$, and of the norm in the product space $\HH$, we get
$$
\begin{aligned}
& \left\|F(X_1)-F(X_2)\right\|_\HH + \left\|\left(G(X_1)-G(X_2)\right)Q^{\frac{1}{2}}\right\|_{\mathcal{L}_2(\HH)}\\
& \quad =  \left\|f(u_1)-f(u_2)\right\|_{{-1}}+\left\|\left(g(u_1)-g(u_2)\right)Q^{\frac{1}{2}}\right\|_{\mathcal{L}_2(H^{-1})}
\\
& \quad \le \left\|\left(I-\Delta_{\mathbb{S}^2}\right)^{-\frac{1}{2}}\right\|_{\mathcal{L}(H^0)}
\left\|f(u_1)-f(u_2)\right\|_{L^2(\mathbb{S}^2)}+
 \left\|\left(I-\Delta_{\mathbb{S}^2}\right)^{-\frac{\delta}{2}}\right\|_{\mathcal{L}(H^0)}\left\|\left(g(u_1)-g(u_2)\right)Q^{\frac{1}{2}}\right\|_{\mathcal{L}_2(H^{\delta-1})}.
\end{aligned}
$$
Using Assumption~\ref{ass_compact}, one then obtains the estimates
$$
\begin{aligned}
& \left\|F(X_1)-F(X_2)\right\|_\HH + \left\|\left(G(X_1)-G(X_2)\right)Q^{\frac{1}{2}}\right\|_{\mathcal{L}_2(\HH)}\\
 & \quad \le L_1\left(\left\|\left(I-\Delta_{\mathbb{S}^2}\right)^{-\frac{1}{2}}\right\|_{\mathcal{L}(H^0)}
+\left\|\left(I-\Delta_{\mathbb{S}^2}\right)^{-\frac{\delta}{2}}\right\|_{\mathcal{L}(H^0)}\right)
\left\|u_1-u_2\right\|_{L^2(\IS^2)}
\le \tilde{L}_1\left\|X_1-X_2\right\|_\HH
\end{aligned}
$$
with $\tilde{L}_1=L_1(\|(I-\Delta_{\mathbb{S}^2})^{-\frac{1}{2}}\|_{\mathcal{L}(H^0)}+\|(I-\Delta_{\mathbb{S}^2})^{-\frac{\delta}{2}}\|_{\mathcal{L}(H^0)})$. This shows the first result of the lemma.

Similarly, one shows the second part of the statement as follows
$$
\begin{aligned}
& \left\|F(X)\right\|_\HH + \left\|G(X)Q^{\frac{1}{2}}\right\|_{\mathcal{L}_2(\HH)}\\
&  \quad =  \left\|f(u)\right\|_{-1}+\left\|g(u)Q^{\frac{1}{2}}\right\|_{\mathcal{L}_2(H^{-1})}
\\
& \quad \le \left\|\left(I-\Delta_{\mathbb{S}^2}\right)^{-\frac{1}{2}}\right\|_{\mathcal{L}(H^0)}
\left\|f(u)\right\|_{L^2(\mathbb{S}^2)}+
 \left\|\left(I-\Delta_{\mathbb{S}^2}\right)^{-\frac{\delta}{2}}\right\|_{\mathcal{L}(H^0)}\left\|g(u)Q^{\frac{1}{2}}\right\|_{\mathcal{L}_2(H^{\delta-1})}\\
 & \quad \le L_2\left(\left\|\left(I-\Delta_{\mathbb{S}^2}\right)^{-\frac{1}{2}}\right\|_{\mathcal{L}(H^0)}
+\left\|\left(I-\Delta_{\mathbb{S}^2}\right)^{-\frac{\delta}{2}}\right\|_{\mathcal{L}(H^0)}\right)
\left(1+\left\|u\right\|_{L^2(\IS^2)}\right)
\le \tilde{L}_2\left(1+\left\|X\right\|_\HH\right),
\end{aligned}
$$
which concludes the proof.
\end{proof}

\subsection{Well-posedness and regularity for the exact solution}
After the above preparations, we can establish the following well-posedness result 
for the semilinear stochastic wave equation on the sphere~\eqref{abstract_eq}. 
\begin{thm}
\label{wp_thm}
Given Assumption~\ref{ass_compact}, suppose further that $X_0\in L^{2p}\left(\Omega; \HH\right)$ for some $p\ge 1$. Then, for any $T\in(0,\infty)$, there exists, up to a modification, a unique continuous mild solution~\eqref{mild_form_comp} to the semilinear stochastic wave equation on the sphere~\eqref{abstract_eq} defined on $[0,T]$. Moreover, the mild solution satisfies 
\begin{equation}
\label{bo_00}
\sup_{t\in\IT}\mathbb{E}\bigg[\left\|X(t)\right\|^{2p}_{\HH}\bigg]\le c\left(\left\|X_0\right\|_{L^{2p}(\Omega;\HH)}^{2p}+1\right).
\end{equation}
In addition, if $X_0\in L^{2p}\left(\Omega; {\HH}^\beta(\mathbb{S}^2)\right)$ for some $p\ge 1$ and for some $\beta\ge 0$, then, for $T\in(0,\infty)$, the mild solution~\eqref{mild_form_comp} satisfies
\begin{equation}
\label{bo_01}
\mathbb{E}\bigg[\sup_{t\in\IT}\left\|X(t)\right\|^{2p}_{{\HH}^\gamma}\bigg]\le c\left(\left\|X_0\right\|_{L^{2p}(\Omega;{\HH}^\beta)}^{2p}+1\right),
\end{equation}
with $\gamma \in [0,\min(\beta,\delta,1)]$.
\end{thm}
The existence and uniqueness of the mild solution~\eqref{mild_form_comp} as well as the moment bound~\eqref{bo_00} follow from Lemma~\ref{lip_comp} and \cite[Theorem~7.2]{MR3236753}, see also \cite{MR3484400,MR3353942} in the Euclidean setting. The second bound can be obtained by following the proof of \cite[Proposition~3.1]{MR3353942}.

\section{Spatial semi-discretization}\label{sec-space}
In this section, we consider the spectral method used in \cite{MR4462619} for the linear stochastic wave equation on the sphere and apply it to the semilinear SPDE~\eqref{stoc_wave}, or its equivalent abstract formulation~\eqref{abstract_eq}. We first show properties of the semi-discrete solution and then prove its rate of strong convergence. 

Since the spherical harmonics build a basis for the space $L^2(\IS^2)$, one can express the two components of the exact solution $X(t)$ of the SPDE~\eqref{abstract_eq} as the following series expansions
\begin{align}\label{serieX}
u(t)=\sum_{\ell=0}^\infty\sum_{m=-\ell}^\ell u^{\ell,m}(t)Y_{\ell, m}\quad\text{and}\quad 
v(t)=\sum_{\ell=0}^\infty\sum_{m=-\ell}^\ell v^{\ell,m}(t)Y_{\ell, m}, 
\end{align}
where $t\in\IT$. Here, one has $u^{\ell,m}(t)=\langle u(t), Y_{\ell,m}\rangle_{L^2(\IS^2)}$ and $v^{\ell,m}(t)\langle v(t), Y_{\ell,m}\rangle_{L^2(\IS^2)}$ for the coefficients of the series expansions.
We set $X^{\ell,m}(t)=\left(u^{\ell,m}(t),v^{\ell,m}(t)\right)^T$ for $\ell\in \mathbb{N}_0$ and $m=-\ell,\dots,\ell$ and denote (with a slight abuse of notation)
\begin{equation}
\label{serieXbis}
X(t)=\sum_{\ell=0}^{\infty}\sum_{m=-\ell}^{\ell} {X}^{\ell,m}(t)Y_{\ell,m}.
\end{equation}
In order to numerically approximate solutions to the SPDE~\eqref{abstract_eq} in space, we truncate the series expansion~\eqref{serieX} at a fixed integer $\kappa\ge 1$. The semi-discrete solution is thus given by similar series expansions
\begin{align}\label{sp_appr}
u^\kappa(t)=\sum_{\ell=0}^\kappa\sum_{m=-\ell}^\ell \widehat{u^{\ell,m}}(t)Y_{\ell, m}\quad\text{and}\quad 
v^\kappa(t)=\sum_{\ell=0}^\kappa\sum_{m=-\ell}^\ell \widehat{v^{\ell,m}}(t)Y_{\ell, m}
\end{align}
or 
\begin{equation}
\label{sp_apprbis}
X^\kappa(t)=\sum_{\ell=0}^{\kappa}\sum_{m=-\ell}^{\ell} 
\widehat{X^{\ell,m}}(t)Y_{\ell,m}.
\end{equation}
Similarly, one has $\widehat{u^{\ell,m}}(t)=\langle u^\kappa(t),Y_{\ell,m}\rangle_{L^2(\IS^2)}$ and $\widehat{v^{\ell,m}}(t)=\langle v^\kappa(t),Y_{\ell,m}\rangle_{L^2(\IS^2)}$. 
Observe that the semi-discrete solution~\eqref{sp_apprbis} verifies the stochastic evolution equation 
\begin{equation}
\label{abstract_spectral}
\dd X^\kappa(t)=A^\kappa X^\kappa(t)\dd t+\mathcal{P}_\kappa F(X^\kappa(t))\dd t+\mathcal{P}_\kappa G(X^\kappa(t))\dd W(t), \qquad X^\kappa(0)=\mathcal{P}_\kappa X_0,
\end{equation}
where $\mathcal{P}_\kappa$ is the projection operator $\mathcal{P}_\kappa u=\sum_{\ell=0}^{\kappa}\sum_{m=-\ell}^{\ell} u^{\ell,m}Y_{\ell,m}$, 
for $u=\sum_{\ell=0}^{\infty}\sum_{m=-\ell}^{\ell} u^{\ell,m}Y_{\ell,m}\in L^2(\mathbb{S}^2)$. Observe also that we have an abuse of notation and set $\mathcal{P}_\kappa X=\left(\mathcal{P}_\kappa x_1,\mathcal{P}_\kappa x_2\right)^T$ for $X=(x_1,x_2)^T$. The linear operator $A^\kappa$ is defined as 
$$
A^\kappa=\begin{pmatrix}
0& I_\kappa\\ \Delta^\kappa_{\mathbb{S}^2}&0\end{pmatrix}
$$
with $I_\kappa,\Delta^\kappa_{\IS^2}\colon\mathcal{P}_\kappa L^2(\IS^2) \to\mathcal{P}_\kappa L^2(\IS^2)$ 
defined as $I_\kappa=I\mathcal{P}_\kappa=\mathcal{P}_\kappa I$ and 
$\Delta^\kappa_{\IS^2}=\Delta_{\IS^2}\mathcal{P}_\kappa=\mathcal{P}_\kappa \Delta_{\IS^2}$ with $\mathcal{P}_\kappa L^2(\IS^2) = \text{span}(Y_{\ell,m}:\ell=0,\ldots,\kappa\:\text{and}\: m=-\ell,\ldots,\ell)$. 
Since $I_\kappa=I, \Delta^\kappa_{\IS^2}=\Delta_{\IS^2}$ and $A^\kappa=A$ when acting on elements from the finite-dimensional space $\mathcal{P}_\kappa L^2(\IS^2)$, below, we will use both notations when appropriate.

It is a direct implication that the semi-discrete problem~\eqref{abstract_spectral} has the mild formulation 
\begin{equation}
\label{mild_trunc}
X^\kappa(t)=E_\kappa(t)\mathcal{P}_\kappa X_0+\int_{0}^{t} E_\kappa(t-s)\mathcal{P}_\kappa F(X^\kappa(s))\dd s
+\int_{0}^{t} E_\kappa(t-s)\mathcal{P}_\kappa G(X^\kappa(s))\dd W(s), 
\end{equation}
where the group~$E_\kappa$ is defined as $E_\kappa(t)=E(t)\mathcal{P}_\kappa=\mathcal{P}_\kappa E(t)$ and similarly for the discrete cosine and sine operators. Since $E(t)=E_\kappa(t)$, and simlarly for the cosine and sine operators, when acting on~$\mathcal{P}_\kappa L^2(\IS^2)$, we may use both notations below. 

Note that the operators $\mathcal{P}_\kappa$ and $\Delta_{\mathbb{S}^2}$ commute. Furthermore, we will make use of the following result on the projection operator $\mathcal{P}_\kappa$.
\begin{lem}
\label{lem_p}
For any $s\in\mathbb{R}$, $\kappa\in\mathbb{N}$ and $w\in H^s(\mathbb{S}^2)$, the linear operators $\mathcal{P}_\kappa$ and $I-\mathcal{P}_\kappa$ satisfy
\begin{equation}
\label{lem_p1}
\left\|\mathcal{P}_\kappa w\right\|_{s}\le \left\| w\right\|_{s} \quad  \text{and} \quad
\left\|\left(I-\mathcal{P}_\kappa \right)w\right\|_{s}\le \left\| w\right\|_{s}.
\end{equation}
Furthermore, for any $s\ge 0, \kappa\in\mathbb{N}$ and  $W=(w_1,w_2)^T \in {\HH}^{s}$, the linear operator $I-\mathcal{P}_\kappa$ satisfies
\begin{equation}
\label{lem_p2}
\left\|(I-\mathcal{P}_\kappa)W\right\|_{\HH}\le \kappa^{-s}\left\|W\right\|_{{\HH}^s},
\end{equation}
where, in an abuse of notation, we set $\left(I-\mathcal{P}_\kappa\right)W=\left(\left(I-\mathcal{P}_\kappa\right)w_1,\left(I-\mathcal{P}_\kappa\right)w_2\right)^T$.
\end{lem}

\begin{proof}
We start with the proof of the first claim in~\eqref{lem_p1}.

Let us write $w$ as the series expansion $ w=\sum_{\ell=0}^{\infty}\sum_{m=-\ell}^{\ell}w^{\ell,m}Y_{\ell,m}$. 
The definition of the norms and the orthonormality of the spherical harmonics yield
$$
\begin{aligned}
\left\|\mathcal{P}_\kappa w\right\|^2_{s} 
&= \left\|\left(I-\Delta_{\mathbb{S}^2}\right)^{\frac{s}{2}}\mathcal{P}_\kappa w\right\|_{L^2(\mathbb{S}^2)}^2
=  \left\|\sum_{\ell = 0}^{\kappa}\sum_{m=-\ell}^{\ell}\left(1+\ell(\ell+1)\right)^{\frac{s}{2}}w^{\ell,m}Y_{\ell,m}\right\|_{L^2(\mathbb{S}^2)}^2\\
& =  \sum_{\ell = 0}^{\kappa}\sum_{m=-\ell}^{\ell}\left(1+\ell(\ell+1)\right)^{s}\big|w^{\ell,m}\big|^2
\le \sum_{\ell = 0}^{\infty}\sum_{m=-\ell}^{\ell}\left(1+\ell(\ell+1)\right)^{s}\big|w^{\ell,m}\big|^2
=\|w\|_s^2,
\end{aligned}
$$
which shows the first claim.

The second bound in \eqref{lem_p1} can be shown similarly. The proof is thus omitted. 

We continue with the proof of~\eqref{lem_p2} and expand first $ w_i=\sum_{\ell=0}^{\infty}\sum_{m=-\ell}^{\ell}w_i^{\ell,m}Y_{\ell,m}$, $i=1,2$. Using similar arguments as above, we thus obtain the estimates
$$
\begin{aligned}
\left\|(I-\mathcal{P}_\kappa)W\right\|^2_{\HH} &= \left\|(I-\mathcal{P}_\kappa)w_1\right\|^2_{L^2(\mathbb{S}^2)}+\left\|\left(I-\Delta_{\mathbb{S}^2}\right)^{-\frac{1}{2}}(I-\mathcal{P}_\kappa)w_2\right\|^2_{L^2(\mathbb{S}^2)}\\
&=\sum_{\ell=\kappa+1}^{\infty}\sum_{m=-\ell}^{\ell}\left(1+\ell(\ell+1)\right)^{-s}\left(
\left(1+\ell(\ell+1)\right)^{s}\big|w_1^{l,m}\big|^2+\left(1+\ell(\ell+1)\right)^{s-1}\big|w_2^{l,m}\big|^2\right)\\
&\le  \kappa^{-2s}\left(\left\|\left(I-\mathcal{P}_\kappa\right)w_1\right\|^2_{s}+\left\|\left(I-\mathcal{P}_\kappa\right)w_2\right\|^2_{s-1}\right)
 \le \kappa^{-2s}\left\|W\right\|^2_{\HH^s},
\end{aligned}
$$
where we have used the second estimate in~\eqref{lem_p1} in the last inequality.
\end{proof}

As a first step in the convergence analysis of the spectral method, we state the well-posedness of the semi-discrete solution $X^\kappa$. The proof of this result is similar to the proof of Theorem~\ref{wp_thm}. We thus omit it.
\begin{prop}
\label{wp_trunc}
Given Assumption~\ref{ass_compact} and $X_0\in L^{2p}\left(\Omega; \HH\right)$ for some $p\ge 1$, for any $T\in(0,\infty)$,  there exists, up to modification, a unique mild solution~\eqref{mild_trunc} to the semi-discrete problem~\eqref{abstract_spectral}. Moreover, such a solution satisfies
\begin{equation}
\label{bo_0_H}
\sup_{\kappa \in \N}\sup_{t\in\IT}\ \mathbb{E}\bigg[\left\|X^\kappa(t)\right\|^{2p}_{\HH}\bigg]\le c\left(\left\|X_0\right\|_{L^{2p}(\Omega;\HH)}^{2p}+1\right).
\end{equation}
Furthermore, if $X_0\in L^{2p}\left(\Omega; {\HH}^\beta\right)$, for some $\beta\ge0$ and $p\ge 1$, then, one has the estimates
\begin{equation}
\label{bo_0}
\sup_{\kappa \in \N}\ \mathbb{E}\bigg[\sup_{t\in\IT}\left\|X^\kappa(t)\right\|^{2p}_{{\HH}^\gamma}\bigg]\le c\left(\left\|X_0\right\|_{L^{2p}(\Omega;{\HH}^\beta)}^{2p}+1\right)
\end{equation}
with $\gamma \in [0,\min\left(1,\beta,\delta\right)]$.
\end{prop}

The next step in showing strong convergence of the spectral method is to prove 
time regularity of the first component of the semi-discrete solution~\eqref{sp_appr}.
\begin{prop}
\label{time_reg}
Given Assumption~\ref{ass_compact} and $X_0\in L^{2p}(\Omega; {\HH}^\beta)$, for some $\beta\ge0$ and $p\ge 1$, there exists a positive constant $c$ such that, for any $0\le t_1\le t_2\le T$, the semi-discrete solution $u^\kappa$ satisfies
$$
\sup_{\kappa\in\mathbb{N}} \big\|u^\kappa(t_2)-u^\kappa(t_1)\big\|_{L^{2p}(\Omega; L^2(\mathbb{S}^2))}\le c \left|t_2-t_1\right|^{\min(1,\beta,\delta)}.
$$
\end{prop}

\begin{proof}
Set $\gamma=\min(1,\beta,\delta)$. From the definition of the mild solution~\eqref{mild_trunc}, we can split
$$
\mathbb{E}\left[\big\| u^\kappa(t_2)-u^\kappa(t_1)\big\|^{2p}_{L^2(\mathbb{S}^2)}\right]\le c\left(J_1^\kappa+J_2^\kappa+J_3^\kappa+J_4^\kappa\right),
$$
where 
$$
\begin{aligned}
&J_1^\kappa=\mathbb{E}\left[\big\| \left(C(t_2-t_1)-I\right)u^\kappa(t_1)\big\|^{2p}_{L^2(\mathbb{S}^2)}\right],\quad
J_2^\kappa=\mathbb{E}\left[\big\| \big(-\Delta_{\mathbb{S}^2}\big)^{-\frac{1}{2}}S(t_2-t_1)v^\kappa(t_1)\big\|^{2p}_{L^2(\mathbb{S}^2)}\right],\\
&J_3^\kappa=\mathbb{E}\left[\bigg\| \int_{t_1}^{t_2}(-\Delta_{\mathbb{S}^2})^{-\frac{1}{2}}S(t_2-s)\mathcal{P}_\kappa f(u^\kappa(s))\dd s\bigg\|^{2p}_{L^2(\mathbb{S}^2)}\right],\\
&J_4^\kappa=\mathbb{E}\left[\bigg\| \int_{t_1}^{t_2}(-\Delta_{\mathbb{S}^2})^{-\frac{1}{2}}S(t_2-s)\mathcal{P}_\kappa g(u^\kappa(s))\dd W(s) \bigg\|^{2p}_{L^2(\mathbb{S}^2)}\right].
\end{aligned}
$$
We will now bound these four terms uniformly in~$\kappa$. 

We start with using Lemma~\ref{lem_op2} 
and the uniform boundedness of~$X^\kappa$ in Proposition~\ref{wp_trunc} to obtain
$$
J_1^\kappa
\le c \left|t_2-t_1\right|^{2p \gamma}
\big\|u^\kappa(t_1)\big\|^{2p}_{L^{2p}(\Omega; H^{\gamma}(\mathbb{S}^2))}
\le c  \left|t_2-t_1\right|^{2p \gamma}
\mathbb{E}\big[\sup_{t\in\IT}\big\|X^\kappa(t)\big\|^{2p}_{{\HH}^\gamma(\mathbb{S}^2)}\big]
\le c  \left|t_2-t_1\right|^{2p \gamma}.
$$
For the second term, $J_2^\kappa$, similar arguments give the uniform bound
$$
J_2^\kappa
\le c \left|t_2-t_1\right|^{2p \gamma}
\big\|v^\kappa(t_1)\big\|^{2p}_{L^{2p}(\Omega;H^{\gamma-1}(\mathbb{S}^2))}
\le c \left|t_2-t_1\right|^{2p \gamma}
\mathbb{E}\big[\sup_{t\in\IT}\big\|X^\kappa(t)\big\|^{2p}_{{\HH}^{\gamma}(\mathbb{S}^2)}\big]
\le c  \left|t_2-t_1\right|^{2p \gamma}.
$$
In order to bound the term $J_3^\kappa$ uniformly, we first use the commutativity of the Laplace--Beltrami operator and the projection operator, Lemma~\ref{lem_p}, and H\"older's inequality
$$
\begin{aligned}
J_3^\kappa 
&\le
\mathbb{E}\bigg[\left(\int_{t_1}^{t_2}\big\|(-\Delta_{\mathbb{S}^2})^{-\frac{1}{2}}S(t_2-s)f(u^\kappa(s))\big\|_{L^2(\mathbb{S}^2)}\dd s\right)^{2p}\bigg]\\
&\le |t_2-t_1|^{2p-1}
\mathbb{E}\left[\int_{t_1}^{t_2}\big\|(-\Delta_{\mathbb{S}^2})^{-\frac{1}{2}}S(t_2-s)f(u^\kappa(s))\bigg\|^{2p}_{L^2(\mathbb{S}^2)}\dd s\right]
\le c\left|t_2-t_1\right|^{4p-1}
\mathbb{E}\bigg[\int_{t_1}^{t_2}\big\|f(u^\kappa(s))\big\|^{2p}_{L^2(\mathbb{S}^2)}\dd s\bigg],
\end{aligned}
$$
where we made use of Lemma~\ref{lem_op2} with $q=1$ in the last step. Next, we use the linear growth assumption on the coefficients of the SPDE from Assumption~\ref{ass_compact} and Proposition~\ref{wp_trunc} 
to get the bounds 
$$
\begin{aligned}
J_3^\kappa 
&\le c_1\left|t_2-t_1\right|^{4p}+c_2\left|t_2-t_1\right|^{4p-1}\int_{t_1}^{t_2}
\mathbb{E}\left[\big\|X^\kappa(s)\big\|_{\HH}^{2p}\right]\dd s
\le c \left|t_2-t_1\right|^{4p} \left(1 + \sup_{t\in\IT}\mathbb{E}\left[\big\|X^\kappa(t)\big\|_{\HH}^{2p}\right]\right)\\
& \le c \left|t_2-t_1\right|^{4p}.
\end{aligned}
$$
Finally, bounds for the term $J_4^\kappa$ are obtained using a Burkholder--Davis--Gundy type inequality (see \cite[Prop.~2.12]{K14} and \cite[Theorem~4.36]{MR3236753}), H\"older's inequality, Lemma~\ref{lem_p}, Lemma~\ref{lem_op2} with $q=\min(\delta,1)$ and Assumption~\ref{ass_compact}. Since $\min(\delta-1,0)-(\delta-1)\le 0$, this implies
$$
\begin{aligned}
J_4^\kappa
&\le c \, \mathbb{E}\bigg[\left(\int_{t_1}^{t_2}\bigg\|(-\Delta_{\mathbb{S}^2})^{-\frac{1}{2}}S(t_2-s)g(u^\kappa(s))Q^{\frac{1}{2}}\bigg\|^2_{\mathcal{L}_2(H^0)}\dd s\right)^p\bigg]\\
&\le c |t_2-t_1|^{p-1}\int_{t_1}^{t_2}\mathbb{E}\bigg[\bigg\|(-\Delta_{\mathbb{S}^2})^{-\frac{1}{2}}S(t_2-s)g(u^\kappa(s))Q^{\frac{1}{2}}\bigg\|^{2p}_{\mathcal{L}_2(H^0)}\bigg]\dd s\\
&\le c |t_2-t_1|^{p-1+2p\min(\delta,1)}\int_{t_1}^{t_2}\mathbb{E}\bigg[\bigg\|g(u^\kappa(s))Q^{\frac{1}{2}}\bigg\|^{2p}_{\mathcal{L}_2(H^{\min(\delta-1,0)})}\bigg]\dd s\\
&\le c\bigg\|\big(I-\Delta_{\mathbb{S}^2}\big)^{\frac{\min(\delta-1,0)-(\delta-1)}{2}}\bigg\|^{2p}_{\mathcal{L}(L^2(\mathbb{S}^2))} |t_2-t_1|^{p-1+2p\min(\delta,1)}\int_{t_1}^{t_2}\mathbb{E}\bigg[\bigg\|g(u^\kappa(s))Q^{\frac{1}{2}}\bigg\|^{2p}_{\mathcal{L}_2(H^{\delta-1})}\bigg]\dd s\\
&\le c
|t_2-t_1|^{p-1+2p\min(\delta,1)}\int_{t_1}^{t_2}\mathbb{E}\bigg[1+\bigg\|X^\kappa(s)\bigg\|^{2p}_\HH \bigg]\dd s.
\end{aligned}
$$
This yields, thanks to Proposition~\ref{wp_trunc}, the uniform bound
$$
\sup_{\kappa\in\mathbb{N}} J_4^\kappa\le c |t_2-t_1|^{p+2p\min(\delta,1)}.
$$

Gathering all the obtained estimates gives
$$
\sup_{\kappa\in\mathbb{N}} \mathbb{E}\big[\big\| u^\kappa(t_2)-u^\kappa(t_1)\big\|^{2p}_{L^2(\mathbb{S}^2)}\big]
\le c\left(|t_2-t_1|^{2p\gamma}+|t_2-t_1|^{4p}+ |t_2-t_1|^{p+2p\min(\delta,1)}\right)
$$
and therefore the desired result. 
\end{proof}
We conclude this section with the following result on the strong and almost sure convergence of the spectral method~\eqref{mild_trunc} when applied to the semilinear stochastic wave equation on the sphere~\eqref{abstract_eq}.
\begin{thm}
\label{spa_thm}
Given Assumption~\ref{ass_compact} and $X_0\in L^{2p}(\Omega; {\HH}^\beta(\mathbb{S}^2))$ for some $p\geq1$ and $\beta\ge 0$, there exists a positive constant~$c$ such that, for any $\kappa\in \mathbb{N}$, the error between the spatial approximation~$X^\kappa$ in~\eqref{mild_trunc} and the solution to~\eqref{abstract_eq} is bounded by
\begin{equation*}
\label{strong_space}
\left(\mathbb{E}\left[\sup_{t\in\IT}\big\|X(t)-X^\kappa(t)\big\|^{2p}_\HH\right]\right)^{\frac{1}{2p}}\le c\left(1+\left\|X_0\right\|_{L^{2p}(\Omega; {\HH}^\beta(\mathbb{S}^2))}\right) \kappa^{-\min\left(1,\beta,\delta\right)}.
\end{equation*} 
In addition, if $X_0\in L^{2p}(\Omega; {\HH}^\beta(\mathbb{S}^2))$ for all $p\geq 1$, then asymptotically for any $t\in \IT$
\begin{equation*}
\big\|X(t)-X^\kappa(t)\big\|_{\HH} \le \kappa^{-r},  \qquad \mathbb{P}-\text{a.s.},
\end{equation*}
for any $r< \min(1,\beta,\delta)$. 
\end{thm}

\begin{proof}
Let $\gamma=\min(1,\beta,\delta)$. We start by splitting the error into
\begin{equation*}
\label{decomp_0}
\mathbb{E}\left[\sup_{t\in\IT}\big\|X(t)-X^\kappa(t)\big\|^{2p}_{\HH}\right]
    \le c\left\{ 
        \mathbb{E}\left[\sup_{t\in\IT}\big\|\big(I-\mathcal{P}_\kappa\big)X(t)\big\|^{2p}_{\HH}\right]
        + \mathbb{E}\left[\sup_{t\in\IT}\big\|\mathcal{P}_\kappa X(t)-X^\kappa(t)\big\|^{2p}_{\HH}\right]
    \right\}.
\end{equation*}
For the first term, we apply Lemma~\ref{lem_p} and Proposition~\ref{wp_thm} to obtain
\begin{equation*}
\label{sec_term}
\begin{aligned}
\mathbb{E}\left[\sup_{t\in\IT}\big\|\big(I-\mathcal{P}_\kappa \big)X(t)\big\|^{2p}_{\HH}\right]
    &\le \kappa^{-2p\gamma} \, \mathbb{E}\left[\sup_{t\in\IT}\big\|X(t)\big\|^{2p}_{{\HH}^\gamma}\right]
    \le c \bigg(1+\left\|X_0\right\|^{2p}_{L^{2p}(\Omega; {\HH}^\beta)}\bigg) \kappa^{-2p\gamma}.
\end{aligned}
\end{equation*}
To bound the second term, we observe that
\begin{align*}
    & \left\| \mathcal{P}_\kappa X(t)-X^\kappa(t) \right\|^{2p}_\HH\\
        & \quad \le c \left\{ 
            \left\| \int_{0}^{t}E_\kappa(t-s)\mathcal{P}_\kappa\left(F(X(s))-F(X^\kappa(s))\right)\dd s\right\|_\HH^{2p}
            + \left\| \int_{0}^{t}E(t-s)\mathcal{P}_\kappa\left(G(X(s))-G(X^\kappa(s))\right)\dd W(s)\right\|^{2p}_\HH
        \right\}\\
        & \quad = c \, \left( I_1(t) + I_2(t)\right).
\end{align*}
We bound the term $I_1(t)$ applying H\"older's inequality, and Lemmas~\ref{lem_op}~and~\ref{lem_p}, and Lemma~\ref{lip_comp}. This yields
$$
\begin{aligned}
\E \left[ \sup_{t \in \IT} I_1(t) \right]
&\le c \, \E \left[ \sup_{t \in \IT}\int_{0}^{t}\big\|E(t-s)\mathcal{P}_\kappa \big(F(X(s))-F(X^\kappa(s))\big)\big\|^{2p}_\HH\dd s \right]\\
& \le c \, \E \left[\int_{0}^{T}\big\|F(X(s))-F(X^\kappa(s))\big\|^{2p}_{\HH}\dd s \right]
\le c \int_{0}^{T}\mathbb{E}\left[\sup_{s\in[0,t]} \big\|X(s)-X^\kappa(s)\big\|^{2p}_\HH\right]\dd t.
\end{aligned}
$$

Using the boundedness of the operator~$E$, the Burkholder--Davis--Gundy type inequality from \cite[Theorem~4.36]{MR3236753}, and H\"older's inequality, we obtain for the second term that 
\begin{align*}
   \E\left[\sup_{t\in\IT} I_2(t)\right] 
    & \le \E\left[\sup_{t\in\IT} \|E(t-T)\|_{\mathcal{L}(\HH)}^{2p} \left\| \int_{0}^{t}E(T-s)\mathcal{P}_\kappa\left(G(X(s))-G(X^\kappa(s))\right)\dd W(s)\right\|^{2p}_\HH\right]\\
    & \le c \, \E \left[ \int_0^T \|E(T-s)\mathcal{P}_\kappa\left(G(X(s))-G(X^\kappa(s))\right)Q^{1/2}\|_{\mathcal{L}_2(\HH)}^{2p} \, \dd s \right]
    \le c \int_0^T \E\left[\|X(s) - X^\kappa(s)\|_\HH^{2p}\right] \, \dd s\\
    & \le c \int_0^T \E\left[\sup_{s \in [0,t]} \|X(s) - X^\kappa(s)\|_\HH^{2p}\right] \, \dd t,
\end{align*}
where we applied in the later steps as for~$I_1(t)$ Lemmas~\ref{lem_op}, \ref{lem_p}, and~\ref{lip_comp}.

Finally, inserting the obtained bounds into~\eqref{decomp_0}, we get the estimates
$$
\mathbb{E}\left[\sup_{t\in\IT}\left\|X(t)-X^\kappa(t)\right\|^{2p}_{\HH}\right]
    \le c_1\left(1+\left\|X_0\right\|^{2p}_{L^{2p}(\Omega; {\HH}^\beta)}\right) \kappa^{-2p\gamma}
    + c_2  \int_{0}^{T}\mathbb{E}\left[\sup_{s\in[0,t]}\left\|X(s)-X^\kappa(s)\right\|^{2p}_\HH\right]\dd t.
$$
Since all expressions are finite by Propositions~\ref{wp_thm} and~\ref{wp_trunc}, an application of Gr\"onwall's inequality then yields the claimed $L^{2p}$ error estimate.

The bound on $\mathbb{P}$-a.s.\ convergence follows from the $p$ independent strong convergence rates. Indeed, for any $r<\min(1,\beta,\delta)$, Chebyschev's inequality yields
$$
\mathbb{P}\left(\left\|X(t)-X^\kappa(t) \right\|_{\HH}\ge \kappa^{-r}\right)
\le \kappa^{2p r}\mathbb{E}\left[\sup_{t\in\IT}\left\|X(t)-X^\kappa(t) \right\|^{2p}_{\HH}\right]
\le  c \kappa^{2p r} \kappa^{-2p\gamma}.
$$
For all $2 p\ge 1/(|\gamma-r|)$, the series $\sum_{\kappa=0}^{\infty}\kappa^{(r-\gamma)2 p}< +\infty$ converges, and an application of the Borel--Cantelli lemma concludes the proof.
\end{proof}

\section{Temporal discretization}\label{sec-time}

Now, we are ready to provide a temporal discretization of the semi-discrete solutions $X^\kappa(t)$ in \eqref{mild_trunc}, for $\kappa\in\mathbb{N}$. In this section, we propose and study the following stochastic exponential integrator also called stochastic trigonometric integrator in the present setting:
\begin{equation}
\label{exp_met}
X_n^\kappa =E(h)X_{n-1}^{\kappa}+h E(h)\mathcal{P}_\kappa F(X_{n-1}^\kappa)+E(h) \mathcal{P}_\kappa G(X_{n-1}^\kappa) \Delta W_{n-1}, \quad n=1,2,\dots,N,
\end{equation}
where $Nh=T$, with time step size $h>0$ and $\Delta W_{n-1}=W(t_n)-W(t_{n-1})$. By iteration and since $E$ is a $C_0$-group, the time integrator~\eqref{exp_met} can be rewritten as 
\begin{equation}
\label{exp_met_comp}
X_n^\kappa 
    = E(t_n)\mathcal{P}_\kappa X_0
        +\int_{0}^{t_n} E(t_n-\floor{s/h}h)\mathcal{P}_\kappa F\big(X^\kappa_{\floor{s/h}}\big)\,\dd s
        +\int_{0}^{t_n} E(t_n-\floor{s/h}h)\mathcal{P}_\kappa G(X^\kappa_{\floor{s/h}})\, \dd W(s),
\end{equation}
where $t_n = nh$ and $\floor{s/h}$ denotes the integer part of~$s/h$.

The following result provides a uniform bound for the numerical solution~\eqref{exp_met} of the stochastic wave equation on the sphere~\eqref{abstract_eq}. This result is shown by following the proof of Theorem~\ref{wp_thm} and \cite[Proposition~3.1]{MR3353942}. Its proof is therefore omitted.
\begin{prop}
\label{prop_bound_fully_discr}
Given Assumption~\ref{ass_compact} and $X_0\in L^{2p}\left(\Omega; {\HH}^\beta\right)$, for some $\beta\ge0$ and $p\ge 1$, the fully discrete approximation~\eqref{exp_met_comp} is uniformly bounded by
\begin{equation*}
\sup_{\kappa\in\mathbb{N}} \sup_{N\in\mathbb{N}} \ \mathbb{E}\bigg[\max_{n=1,\dots,N}\left\|X^\kappa_{n}\right\|^{2p}_{\HH^\gamma}\bigg]\le c\bigg(\left\|X_0\right\|^{2p}_{L^{2p}(\Omega; \HH^\beta)}+1\bigg),
\end{equation*}
with $\gamma \in [0,\min(\beta,\delta,1)]$.
\end{prop}
    
We now state and prove the main result of the paper: the strong and almost-sure rates of convergence of the stochastic trigonometric integrator~\eqref{exp_met_comp}.
\begin{thm}
\label{tem_thm}
Given Assumption~\ref{ass_compact} and $X_0\in L^{2p}(\Omega; \HH^\beta)$ with $\beta \ge 0$ and $p\ge 1$, the error between the fully discrete solution~\eqref{exp_met_comp} and the spatial approximation~\eqref{mild_trunc} is uniformly bounded by
$$
\sup_{\kappa \in \N}\ \bigg(\mathbb{E}\bigg[\max_{n=1,\dots,N} \big\|X^\kappa(t_n)-X^\kappa_n\big\|^{2p}_{\HH}\bigg]\bigg)^{\frac{1}{2p}}
\le c h^{\min\left(1,\beta,\delta\right)}.
$$
In addition, if $X_0\in L^{2p}(\Omega; {\HH}^\beta(\mathbb{S}^2))$ for all $p\geq 1$, then asymptotically for any $0\leq t_n\leq T$
\begin{equation*}
\big\|X^\kappa(t_n)-X^\kappa_n \big\|_{\HH} \le h^{r},  \qquad \mathbb{P}-\text{a.s.},
\end{equation*}
for any $r< \min(1,\beta,\delta)$.
\end{thm}

\begin{proof}
Using the mild formulations of the semi-discrete and fully discrete solutions, \eqref{mild_trunc} and \eqref{exp_met_comp}, we split the error in six terms that we bound separately
$$
\mathbb{E}\bigg[\max_{n=1,\dots,N}\big\|X^\kappa(t_n)-X_{n}^\kappa \big\|^{2p}_{\HH}\bigg]\le c \left(J_1+J_2+J_3+I_1+I_2+I_3\right),
$$
where
$$
\begin{aligned}
   J_1 & =\mathbb{E}\left[\max_{n=1,\dots,N}\bigg\|
    \int_{0}^{t_n} E(t_n-s)
\mathcal{P}_\kappa\bigg( F(X^\kappa(s))- F\big(X^\kappa(\floor{s/h}h)\big)\bigg)\,\dd s\bigg\|^{2p}_{\HH}\right],\\
J_2 & =\mathbb{E}\left[\max_{n=1,\dots,N}\bigg\|
\int_{0}^{t_{n}}\big(E(t_n-s)-E(t_n-\floor{s/h}h)\big)\mathcal{P}_\kappa F\big(X^\kappa(\floor{s/h}h)\big)\,\dd s\bigg\|^{2p}_{\HH}\right],\\
J_3 & =\mathbb{E}\left[\max_{n=1,\dots,N}\bigg\|\int_{0}^{t_{n}}E(t_n-\floor{s/h}h)
\mathcal{P}_\kappa \big(F\big(X^\kappa(\floor{s/h}h)\big)-F(X^\kappa_{\floor{s/h}})\big)\,\dd s\bigg\|^{2p}_{\HH}\right],\\
I_1 & =\mathbb{E}\left[\max_{n=1,\dots,N}\bigg\|
    \int_{0}^{t_n} E(t_n-s)
\mathcal{P}_\kappa \left(G(X^\kappa(s))- G\big(X^\kappa(\floor{s/h}h))\right)\, \dd W(s)\bigg\|^{2p}_{\HH}\right],\\
I_2 & =\mathbb{E}\left[\max_{n=1,\dots,N}\bigg\|
\int_{0}^{t_{n}}\big(E(t_n-s)-E(t_n-\floor{s/h}h)\big)\mathcal{P}_\kappa \big(G\big(X^\kappa(\floor{s/h}h)\big)\,\dd W(s)\big)\,\bigg\|^{2p}_{\HH}\right],\\
I_3 & =\mathbb{E}\left[\max_{n=1,\dots,N}\bigg\|\int_{0}^{t_{n}}E(t_n-\floor{s/h}h)
\mathcal{P}_\kappa \big(\big(G\big(X^\kappa(\floor{s/h}h)\big)-G(X^\kappa_{\floor{s/h}})\big)\,\dd W(s)\big)\bigg\|^{2p}_{\HH}\right].
\end{aligned}
$$

Since $\|(I-\Delta_{\IS^2})^{-r}\|_{\mathcal{L}(H^0)} \le 1$ for all $r>0$, the first term~$J_1$ is uniformly bounded in~$\kappa$ by H\"older's inequality,  Lemma~\ref{lem_op}, Lemma~\ref{lem_p}, Assumption~\ref{ass_compact}, and Proposition~\ref{time_reg}:
$$
\begin{aligned}
J_1& \le c \, \mathbb{E}\bigg[\max_{n=1,\dots,N}
    \int_{0}^{t_n} \bigg\|f(u^\kappa(s))-f(u^\kappa(\floor{s/h}h))\bigg\|^{2p}_{-1}\,\dd s\bigg]\\
    &\le c \big\|\big(I-\Delta_{\mathbb{S}^2}\big)^{-\frac{1}{2}} \big\|^{2p}_{\mathcal{L}(H^0)}
    \mathbb{E}\bigg[\max_{n=1,\dots,N}
    \int_{0}^{t_n} \bigg\|f(u^\kappa(s))-f(u^\kappa(\floor{s/h}h))\bigg\|^{2p}_{L^{2}(\mathbb{S}^2)}\,\dd s\bigg]\\
    &\le c \int_{0}^{T} \mathbb{E}\bigg[\bigg\|u^\kappa(s)-u^\kappa(\floor{s/h}h)\bigg\|^{2p}_{L^{2}(\mathbb{S}^2)}\bigg]\,\dd s\\
  &\le c h^{2p\ \min(\beta,\delta,1)}.
\end{aligned}
$$
To bound $J_2$, we first observe that 
\begin{equation}
\label{eqn:bound_E}    
\|E(t_n-s)-E(t_n-\floor{s/h}h)\|_{\mathcal{L}(\HH^{\min(\beta,\delta,1)};\HH)}
    \le \|E(t_n - \floor{s/h}h)\|_{\mathcal{L}(\HH)} \|E(\floor{s/h}h -s) - I \|_{\mathcal{L}(\HH^{\min(\beta,\delta,1)};\HH)}
    \le c h^{\min(\beta,\delta,1)}
\end{equation}
by Lemmas~\ref{lem_op} and~\ref{lem_op2}.
Therefore, together with H\"older's inequality, Lemma~\ref{lem_p}, Assumption~\ref{ass_compact}, and Proposition~\ref{wp_trunc}, we obtain
$$
\begin{aligned}
J_2 &\le c \, \mathbb{E}\bigg[\max_{n=1,\dots,N}
\int_{0}^{t_{n}} \bigg\| \big(E(t_n-s)-E(t_n-\floor{s/h}h)\big)\mathcal{P}_\kappa F\big(X^\kappa(\floor{s/h}h)\big)\bigg\|^{2p}_{\HH}\,\dd s\bigg]\\
&\le c h^{2p \ \min(\beta,\delta,1)} \mathbb{E}\bigg[
\int_{0}^{T} \bigg\| \mathcal{P}_\kappa F\big(X^\kappa(\floor{s/h}h)\big)\bigg\|^{2p}_{\HH^{\min(\beta,\delta,1)}}\,\dd s\bigg]\\
&\le c h^{2p \ \min(\beta,\delta,1)} 
\big\|\big(I-\Delta_{\mathbb{S}^2}\big)^{\frac{\min(\beta,\delta,1)-1}{2}}\big\|^{2p}_{\mathcal{L}(H^0)} \mathbb{E}\bigg[
\int_{0}^{T} \bigg\| f\big(u^\kappa(\floor{s/h}h)\big)\bigg\|^{2p}_{L^{2}(\mathbb{S}^2)}\,\dd s\bigg]\\
&\le c h^{2p \ \min(\beta,\delta,1)}\left(1+\mathbb{E}\bigg[\int_{0}^{T}\big\|X^\kappa(\floor{s/h}h)\big\|^{2p}_{\HH}\,\dd s\bigg]\right)\\
&\le c h^{2p \ \min(\beta,\delta,1)}.
\end{aligned}
$$

Similar computations as above yield for~$J_3$ the bounds
$$
\begin{aligned}
J_3& \le c \, \mathbb{E}\left[\int_{0}^{T}\big\| F\big(X^\kappa(\floor{s/h}h)\big)-F(X^\kappa_{\floor{s/h}})\big\|^{2p}_{\HH}\,\dd s\right]
\le  c h  
\sum_{j=0}^{N-1}\mathbb{E}\left[\max_{0\le l \le j}\big\| X^\kappa(t_l)-X^\kappa_{l}\big\|^{2p}_{\HH}\right].
\end{aligned}
$$
The term $I_1$ is bounded with a Burkholder--Davis--Gundy type inequality in the same way as $I_2(t)$ in the proof of Theorem~\ref{spa_thm}, which yields together with Assumption~\ref{ass_compact} and Proposition~\ref{time_reg}
\begin{align*}
   I_1 
    & \le c  \int_0^T \E \left[ \|(G(X^\kappa(s))-G(X^\kappa(\floor{s/h}h)))Q^{1/2}\|_{\mathcal{L}_2(\HH)}^{2p}\right] \, \dd s \\
    & \le c \big\|\big(I-\Delta_{\mathbb{S}^2}\big)^{-\frac{\delta}{2}} \big\|^{2p}_{\mathcal{L}(H^0)}
\int_{0}^{T} \E[ \| (g(u^\kappa(s))-g(u^\kappa(\floor{s/h}h)))Q^{\frac{1}{2}}\|^{2p}_{\mathcal{L}_2(H^{\delta-1})}] \,\dd s\\
    & \le c \sum_{j=0}^{N-1}\int_{t_j}^{t_{j+1}} \E[\|u^\kappa(s)-u^\kappa(t_j)\|^{2p}_{L^2(\mathbb{S}^2)}] \,\dd s
    \le c h^{2p \ \min(\beta,\delta,1)}.
\end{align*}
Combining the Burkholder--Davis--Gundy type inequality with~\eqref{eqn:bound_E} on~$\HH^{\min(\delta,1)}$ yields with Lemmas~\ref{lem_op} and~\ref{lem_p} for $I_2$
\begin{align*}
    I_2
    & \le \E\left[ \max_{n=1,\dots,N} \|E(t_n - T)\|_{\mathcal{L}(\HH)}^{2p} \bigg\|\int_{0}^{t_{n}}E(T-\floor{s/h}h)(E(\floor{s/h}h-s)-I)\mathcal{P}_\kappa G\big(X^\kappa(\floor{s/h}h)\big)\,\dd W(s)\,\bigg\|^{2p}_{\HH}\right]\\
    & \le c h^{2p \min(\delta,1)} \int_0^T \big\|G(X^\kappa(\floor{s/h}h))Q^{\frac{1}{2}}\big\|_{\mathcal{L}_2(\HH^{\min(\delta,1)})}^{2p} \, \dd s,
\end{align*}
where the last integral is uniformly bounded by Assumption~\ref{ass_compact} and Proposition~\ref{wp_trunc}.

Bounding $I_3$ in the same way as $I_2$ in the proof of Theorem~\ref{spa_thm}, we obtain
\begin{align*}
    I_3
    & \le c \int_0^T \E\left[\|X^\kappa(\floor{s/h}h) - X^\kappa_{\floor{s/h}}\|_\HH^{2p}\right] \, \dd s
    \le c h  \sum_{j=0}^{N-1}\mathbb{E}\left[\max_{0\le l \le j}\big\| X^\kappa(t_l)-X^\kappa_{l}\big\|^{2p}_{\HH}\right].
\end{align*}
 
Finally, gathering all the above estimates, we end up with the estimates
$$
\mathbb{E}\bigg[\max_{n=1,\dots,N}\big\|X^\kappa(t_n)-X_{n}^\kappa \big\|^{2p}_{\HH}\bigg]
    \le 
        c h^{2p\min(\beta,\delta,1)}
        + c h \sum_{j=0}^{N-1}\mathbb{E}\bigg[\max_{0\le l \le j}\big\|
 X^\kappa(t_j)-X^\kappa_{j}\big\|^{2p}_{\HH}\bigg].
$$
Since all quantities are finite by Propositions~\ref{wp_trunc} and~\ref{prop_bound_fully_discr} and uniform in~$\kappa$, an application of a discrete Gr\"onwall inequality as in \cite[Lemma~A.4]{K14} yields the claim.

With the same arguments as in the proof of Theorem~\ref{spa_thm}, one can show $\mathbb{P}$-a.s.\ convergence of the fully discrete solution. This concludes the proof of the theorem.
\end{proof}

Having Theorem~\ref{spa_thm} and Theorem~\ref{tem_thm} available, the convergence of the fully discrete scheme to the solution~$X$ of the semilinear stochastic wave equation~\eqref{abstract_eq}  follows by the triangle inequality.
\begin{cor}
\label{ful_res}
Under the assumptions of Theorems~\ref{spa_thm} and~\ref{tem_thm} with $\kappa=N$ and $h=T/N$, the error of the fully discrete approximation~\eqref{exp_met_comp} of the semilinear stochastic wave equation~\eqref{abstract_eq} is bounded by
$$
\bigg(\mathbb{E}\bigg[\max_{n=1,\dots,N}\big\|X(t_n)-X^\kappa_{n}\big\|^{2p}_{\HH}\bigg]\bigg)
\le c \left(\kappa^{-\min(1,\beta,\delta)} + h^{\min(1,\beta,\delta)}\right)
\le c N^{-\min(1,\beta,\delta)}.
$$
In addition, if $X_0\in L^{2p}(\Omega; {\HH}^\beta(\mathbb{S}^2))$ for all $p\geq 1$, then asymptotically for any $0\leq t_n\leq T$
\begin{equation*}
\big\|X(t_n)-X^\kappa_n \big\|_{\HH} \le (\kappa^{-1} + h)^{r},  \qquad \mathbb{P}-\text{a.s.},
\end{equation*}
for any $r< \min(1,\beta,\delta)$.
\end{cor}

\section{Numerical experiments}\label{sec-num}

We conclude the paper with several numerical experiments to illustrate and confirm the strong and almost-sure rates of convergence obtained in Theorems~\ref{spa_thm}~and~\ref{tem_thm}. We start with the case of a stochastic wave equation driven by additive noise and close with a problem with multiplicative noise. The code that was used in this section is available at~\cite{CDL26_code}.

\subsection{The additive case}
We consider the seminlinear stochastic wave equation on the sphere~\eqref{stoc_wave} for the time interval $[0,T]$, with $T=1$, and with additive noise and the following nonlinearity 
\begin{equation}
\label{op_test1}
f(u)=\sin(u^{0,0})Y_{0,0}+\sum_{\ell=1}^{\infty}\sin(u^{\ell,0})Y_{\ell,0}+\sum_{\substack{m=-\ell\\ m\ne 0}}^{\ell}\left(\sin({\rm Re}u^{\ell,m})+i \sin({\rm Im}u^{\ell,m})\right)Y_{\ell,m}.
\end{equation}
The initial values are taken to be 
\begin{equation}
\label{in_data1}
u_0=\sum_{\ell=1}^{\infty}\ell^{-\gamma}Y_{\ell,0}, \quad v_0=\sum_{\ell=1}^{\infty}\ell^{-(\gamma-1)}Y_{\ell,0},
\end{equation}
where $\gamma>\beta+\frac12$ is chosen such that $X_0=(u_0, v_0)^T\in \HH^{\beta}(\mathbb{S}^2)$. 
Let $\delta\geq0$, the eigenvalues $A_{\ell}, \ell=1,2,\dots$, of the covariance operator are taken to be 
\begin{equation}
\label{eigen_q}
A_\ell=\begin{cases}
    1,& \ell=0,\\
    \ell^{-\alpha}, & \ell=1,2,\dots,
\end{cases}
\end{equation}
for some $\alpha>2\delta$. This yields $Q^{\frac{1}{2}}\in\mathcal{L}_2(H^{\delta-1})$, see e.g., \cite{lang2025approximationlevydrivenstochasticheat} for details. 

In the above setting, Assumption~\ref{ass_compact} is verified for the nonlinearity~\eqref{op_test1} and covariance operator $Q$ with eigenvalues given in \eqref{eigen_q}. 

Let us first illustrate the behavior of solutions to the considered SPDE. In Figure~\ref{sample_add}, we display a sample of the solution to the semilinear  stochastic wave equation on the sphere with additive noise, with eigenvalues as in \eqref{eigen_q}, nonlinearity \eqref{op_test1} and initial data \eqref{in_data1}. Here, we take $\delta=\beta=1$ and $\alpha=2\delta+1e-06$. Moreover, the solution has been computed with the stochastic trigonometric method~\eqref{exp_met} with  time step size $h=2^{-14}$ and the spectral method~\eqref{sp_appr} with the truncation parameter $\kappa=2^7$.

\begin{figure}
\centering
\subfigure{\includegraphics[width=0.48\textwidth]{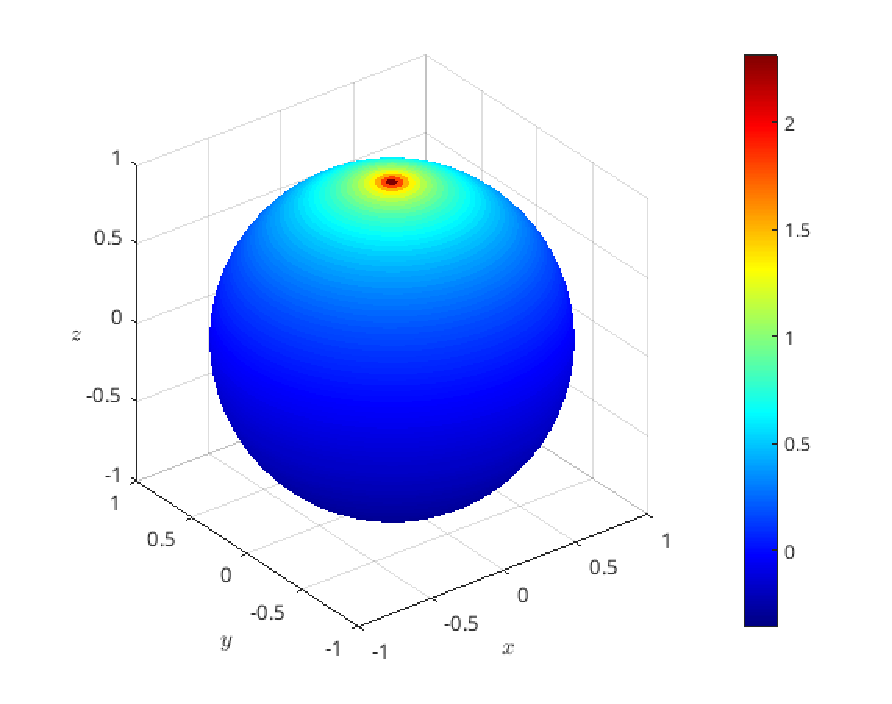}}\quad \subfigure{\includegraphics[width=0.48\textwidth]{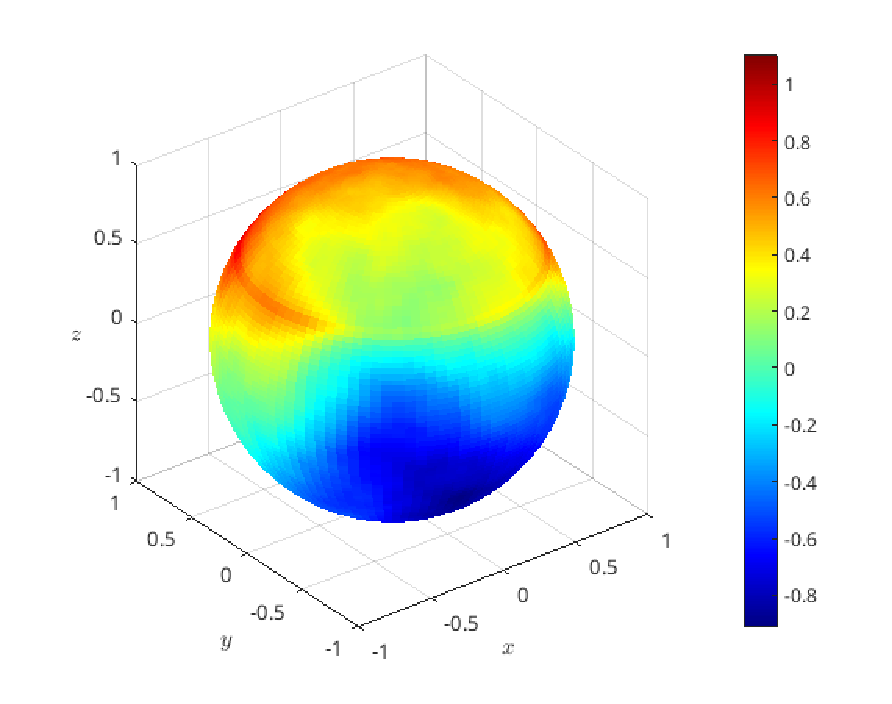}}
          \caption{Sample paths (at time $0$ (left) and at time $1$ (right)) of the solution to the stochastic wave equation on the sphere with additive noise \eqref{eigen_q}, nonlinearity \eqref{op_test1} and initial data \eqref{in_data1} with $\beta=1$. Here, we take $\alpha=2+1e-6$.}
           \label{sample_add}
\end{figure}

We continue by investigating the spatial rate of convergence of the spectral discretization~\eqref{sp_appr} and illustrate the rate of spatial convergence stated in Theorem~\ref{spa_thm}. We apply the time integrator \eqref{exp_met} with a fixed time step size $h=2^{-10}$ and vary the truncation index $\kappa=2^j$, for $j=1,\dots,7$. The errors for the position are measured in the $L^2(\Omega; L^2(\mathbb{S}^2))$-norm, the ones for the velocity in the $L^2(\Omega; H^{-1}(\mathbb{S}^2))$-norm. Furthermore, we also compute pathwise errors in this numerical experiments. The reference solution is computed with the same numerical method and with the discretization parameter $\kappa_{\text{ref}}=2^{9}$. We used $M=100$ independent samples to approximate the expectations, and we have verified that this is enough for the Monte Carlo error to be negligible. Observe that these errors and number of samples are also used in the numerical experiments below. The results for the parameters $\beta,\delta=1,1/2,1/4$, with $\alpha=2\delta+1e-6$, are presented in Figure~\ref{fig1}. Rates of convergence $1, 1/2$, or $1/4$, depending on the choice of these parameters, are observed. This confirms the results of Theorem~\ref{spa_thm} on the spatial rate of convergence of the proposed numerical scheme. 

\begin{figure}
\centering
\subfigure{\includegraphics[width=0.45\textwidth]{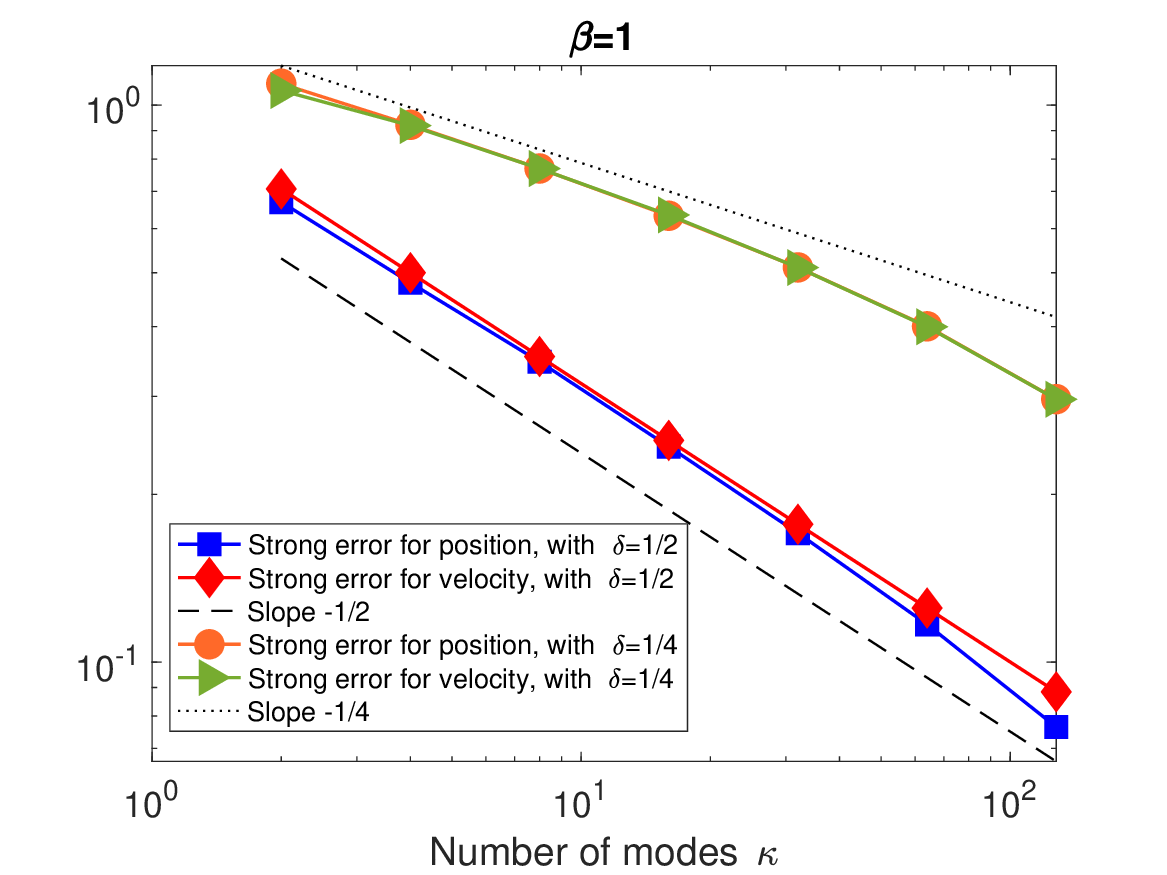}}\quad \subfigure{\includegraphics[width=0.45\textwidth]{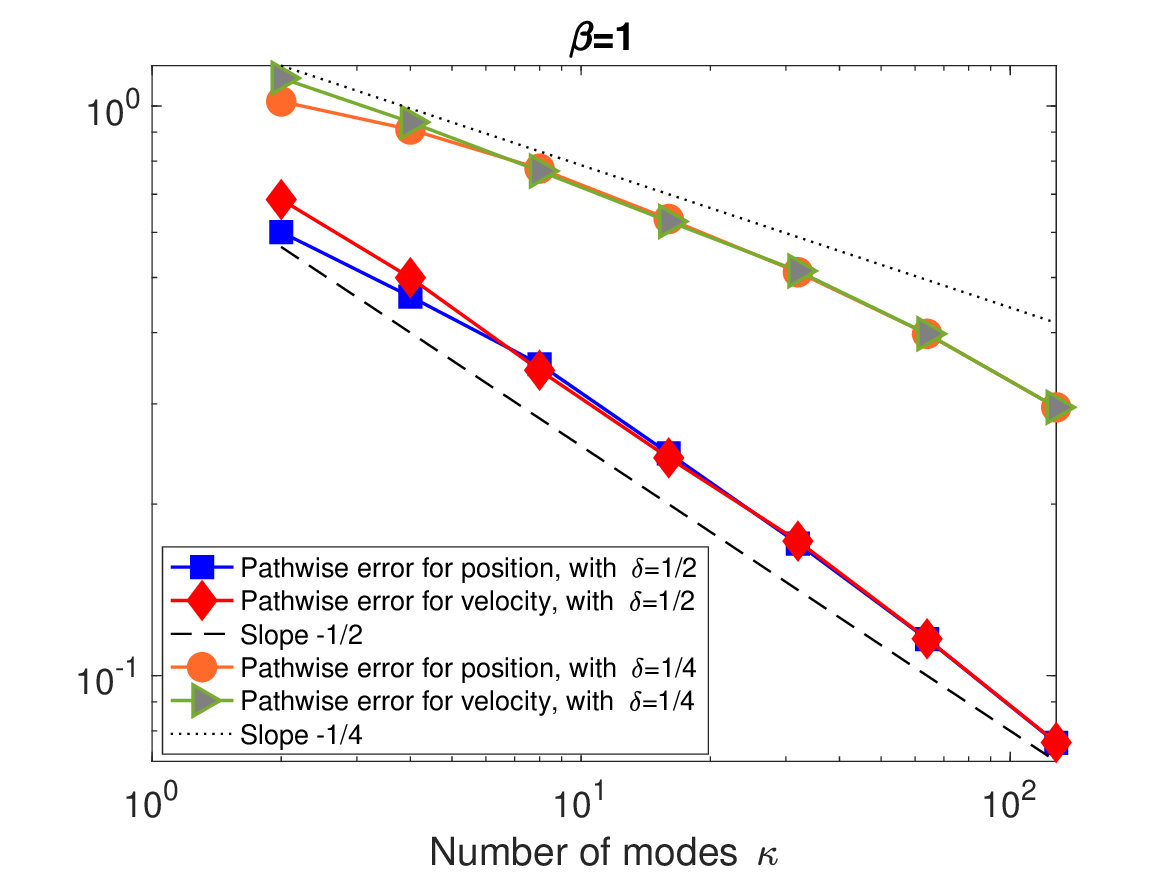}}\\
\subfigure{\includegraphics[width=0.45\textwidth]{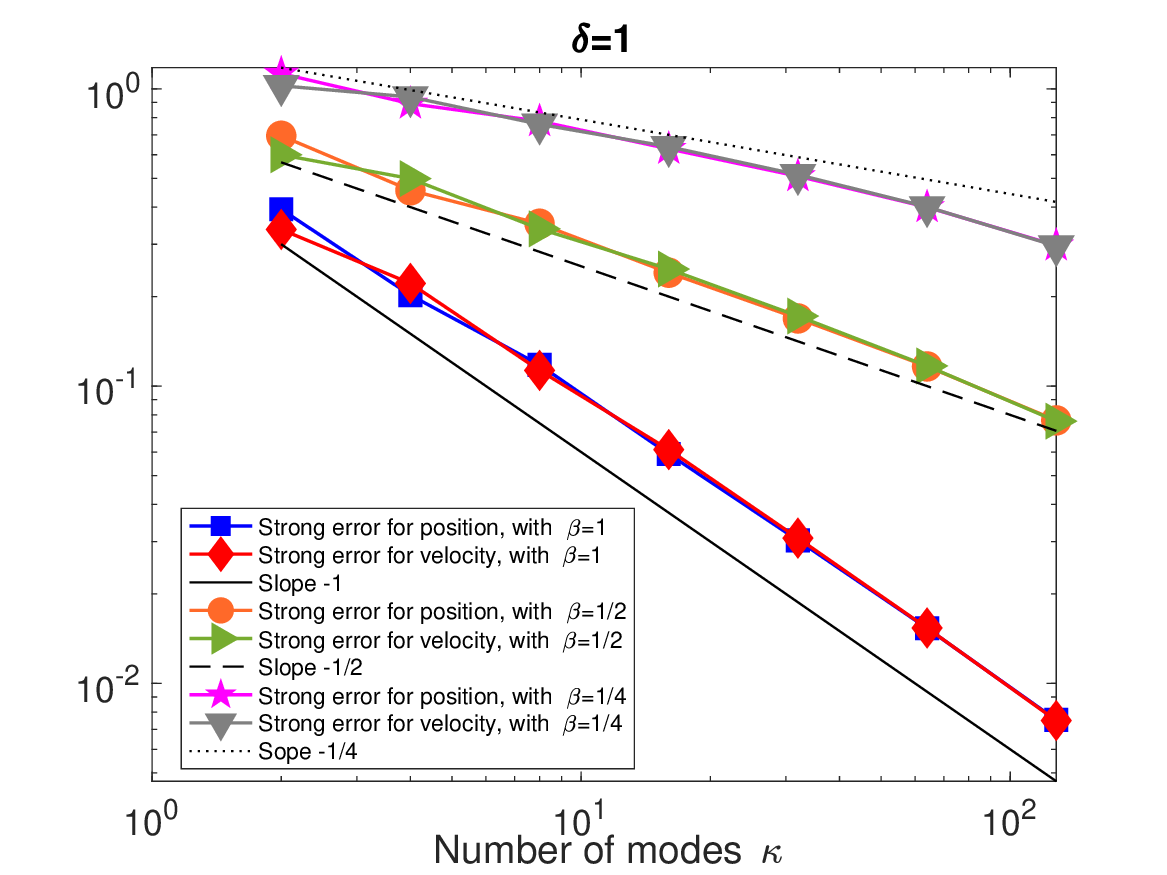}}\quad \subfigure{\includegraphics[width=0.45\textwidth]{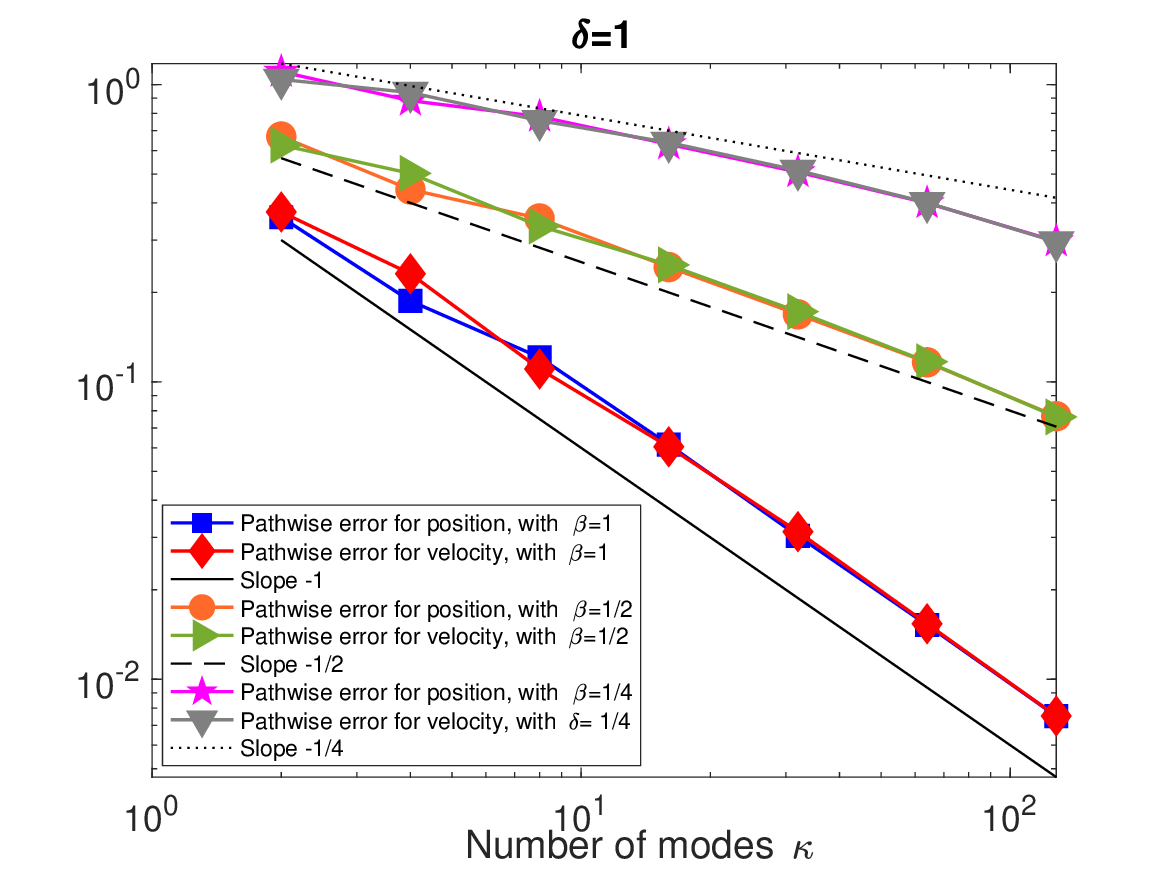}}
          \caption{Convergence in space: Strong errors (on the left) and pathwise errors (on the right) for different values of $\beta$ and $\delta$ and for the SPDE~\eqref{stoc_wave} with additive noise represented by~\eqref{eigen_q}, nonlinearity~\eqref{op_test1} and initial values~\eqref{in_data1}.}
           \label{fig1}
\end{figure}

We now investigate the temporal rate of convergence of the stochastic trigonometric integrator~\eqref{exp_met}, denoted by \textsc{STM} below. We consider the semilinear stochastic wave equation on the sphere~\eqref{stoc_wave} with additive noise, with the operator $Q$ as in \eqref{eigen_q} and the nonlinearity~\eqref{op_test1}. This SPDE is considered on the time interval $[0,T]$ with $T=1$,  with initial values given by \eqref{in_data1}. Furthermore, 
we always take $\alpha=2\delta+1e-6$. In addition, we compare the proposed time integrator with the semi-implicit Euler--Maruyama scheme, denoted by \textsc{SI} below,  
\begin{equation*}
    \label{si_met}
    X^\kappa_{n}=X^\kappa_{n-1}+h A X_n^\kappa+h\mathcal{P}_\kappa F(X_{n-1}^\kappa)+\mathcal{P}_\kappa G(X_{n-1}^\kappa)\Delta W_{n-1}.
\end{equation*}
The reference solution is computed using the \textsc{SI} scheme with time step size $h_{\text{ref}}=2^{-12}$. For both reference and numerical solutions, we consider a fixed truncation parameter $\kappa=2^9$. The results are presented in Figure~\ref{fig4}. The rate of convergence of the \textsc{STM} scheme is observed to be $1$ for $\delta=1$ and $1/2$ for $\delta=1/2$. This illustrates the theoretical rates in Theorem~\ref{tem_thm}. For the \textsc{SI} scheme, 
the observed rates are $1/2$ for $\delta=1$ and between $1/2$ and $1/4$ for $\delta=1/2$. 

\begin{figure}
\centering
\subfigure{\includegraphics[width=0.45\textwidth]{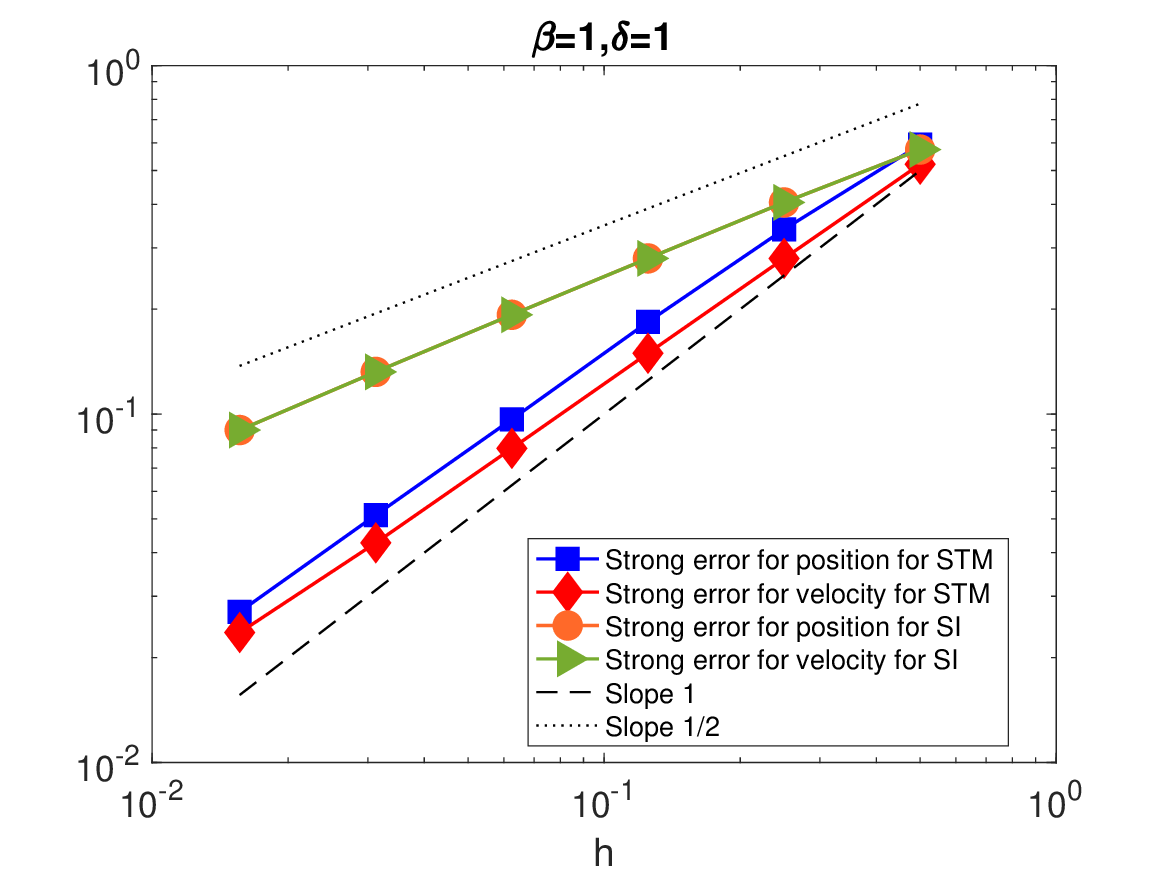}}\quad \subfigure{\includegraphics[width=0.45\textwidth]{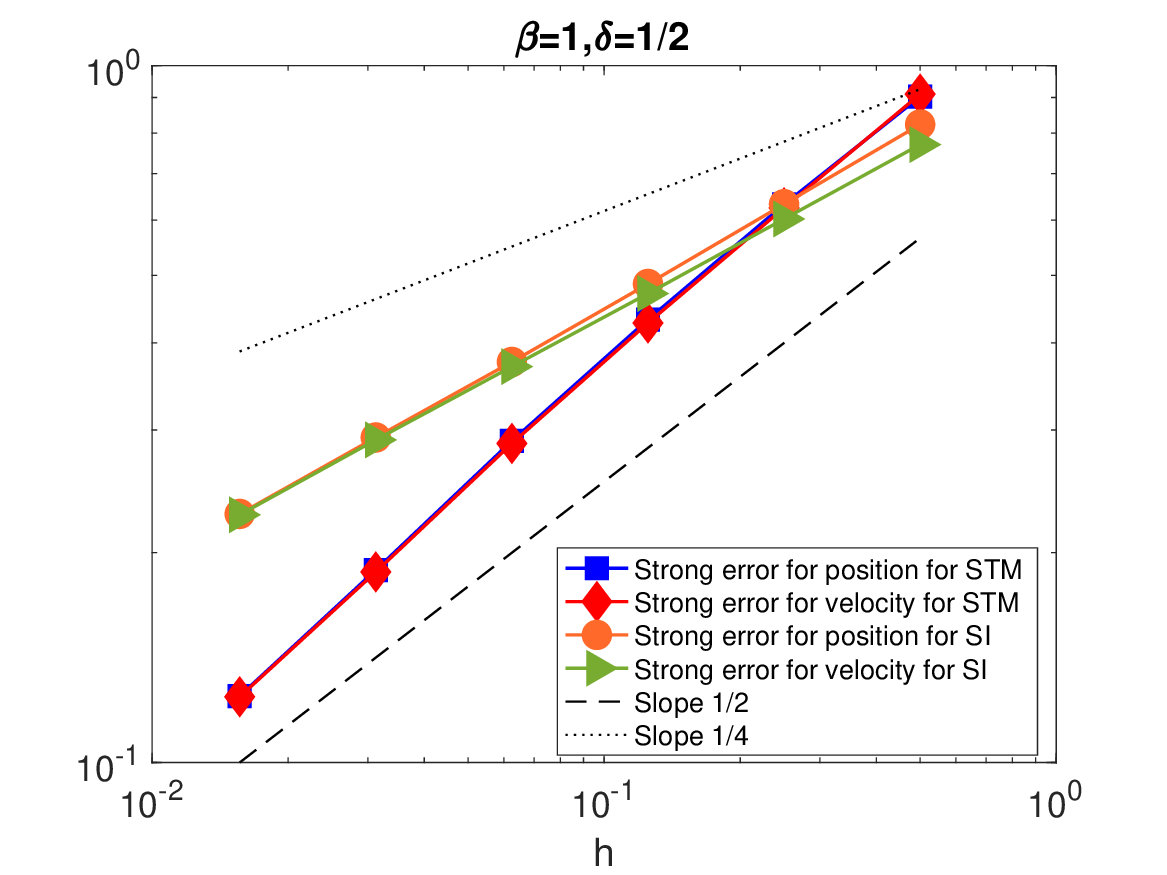}}
          \caption{Convergence in time: $L^2(\Omega; L^2(\mathbb{S}^2))$, resp. $L^2(\Omega; H^{-1}(\mathbb{S}^2))$ errors, for position, resp. velocity, for different values of $\beta$ and $\delta$ and for the SPDE~\eqref{stoc_wave} with additive noise represented by \eqref{eigen_q}, nonlinearity~\eqref{op_test1} and initial values~\eqref{in_data1}.}
           \label{fig4}
\end{figure}

We now proceed with numerically illustrating the pathwise $L^2(\mathbb{S}^2)$ and $H^{-1}(\mathbb{S}^2)$ errors in the position, resp. velocity of the stochastic trigonometric scheme~\eqref{exp_met}. To do this, 
we consider the same SPDE as above and take the discretization parameter $\kappa=2^{9}$.
The reference solution is computed with the stochastic trigonometric integrator~\eqref{exp_met} with a step size $h_{\rm ref}=2^{-11}$. The results are displayed in  Figure~\ref{fig5} for the parameters $\beta=1$ and $\delta=1/4$. A rate $1/4$ of pathwise convergence is observed for the \textsc{STM} method, confirming the result in Theorem~\ref{tem_thm}. For the \textsc{SI} method, a rate $1/8$ is observed.

\begin{figure}
\centering
\subfigure{\includegraphics[width=0.45\textwidth]{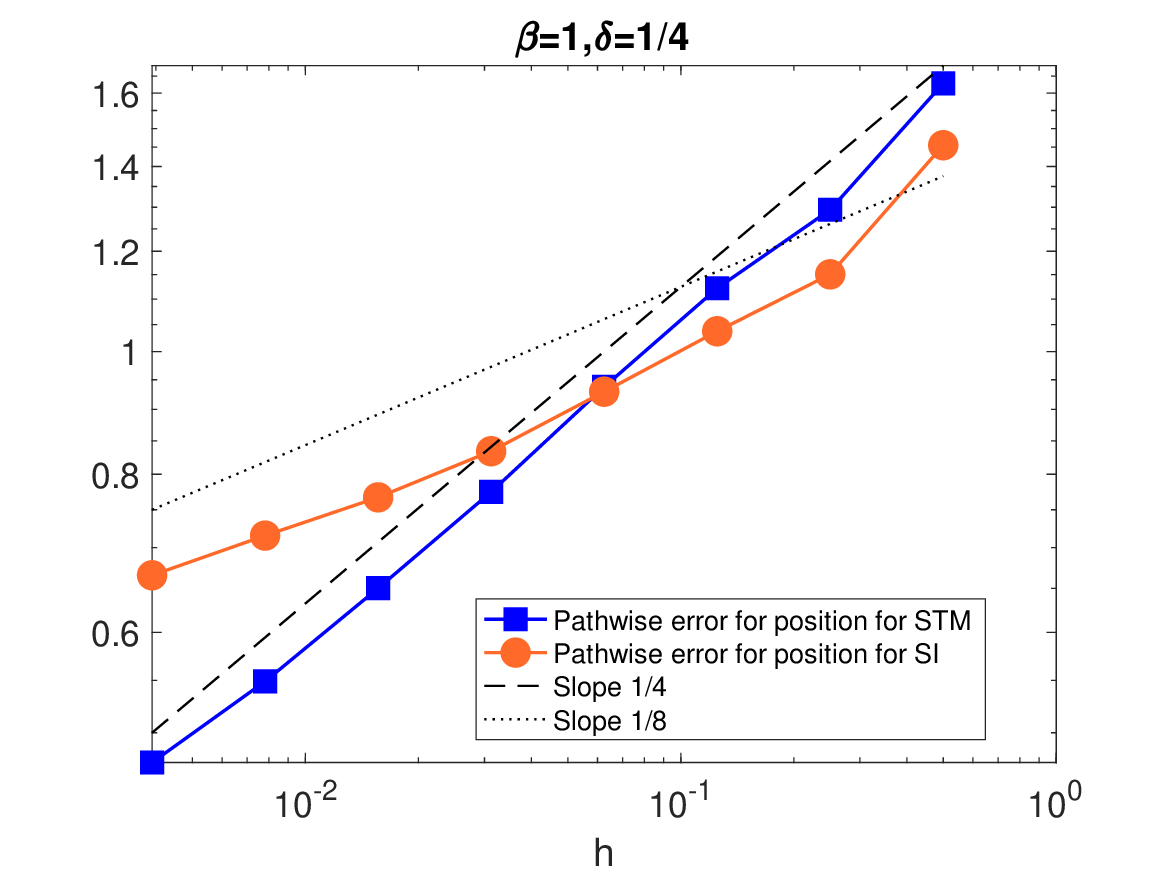}}\quad \subfigure{\includegraphics[width=0.45\textwidth]{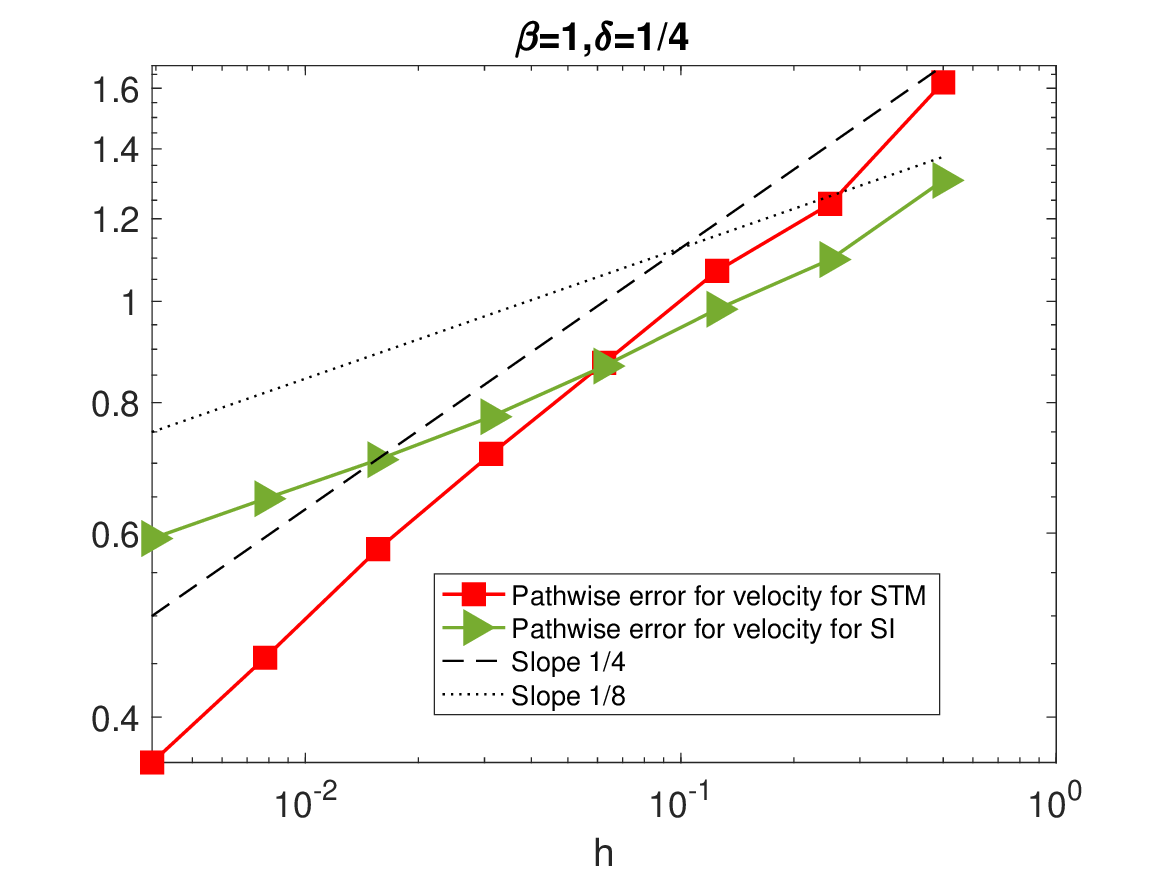}}
          \caption{Convergence in time: Pathwise $L^2(\mathbb{S}^2)$ and $ H^{-1}(\mathbb{S}^2)$ errors for position (on the left) and velocity (on the right) for 
          $\beta=1$, $\delta=1/4$ and for the SPDE~\eqref{stoc_wave} with additive noise represented by~\eqref{eigen_q}, nonlinearity~\eqref{op_test1} and initial values~\eqref{in_data1}.}
           \label{fig5}
\end{figure}

To further illustrate the performance of the stochastic trigonometric integrator \eqref{exp_met}, we include convergence results for nonlinearities in the SPDE~\eqref{stoc_wave} that have no spectral expansion in Figure~\ref{fig6}. These nonlinearities are 
\begin{equation}
f_1(u)=\sin(u),\qquad \qquad f_2(u) = \frac{1+u}{1+u^2}.
\label{testOp2}
\end{equation}
Furthermore, we compute the reference solution with the proposed numerical discretization with parameters $h_{\rm ref}=2^{-8}$ and, for both reference and numerical solutions, we used a fixed truncation parameter $\kappa=2^7$. The results are presented in Figure~\ref{fig6} for the parameters $\beta=\delta=1$ and $\alpha=2+1e-6$. A convergence rate equal to $1$ is observed for the stochastic trigonometric integrator~\eqref{exp_met}. This illustrates the theoretical rates in Theorem~\ref{tem_thm}.

\begin{figure}
\centering
\subfigure{\includegraphics[width=0.45\textwidth]{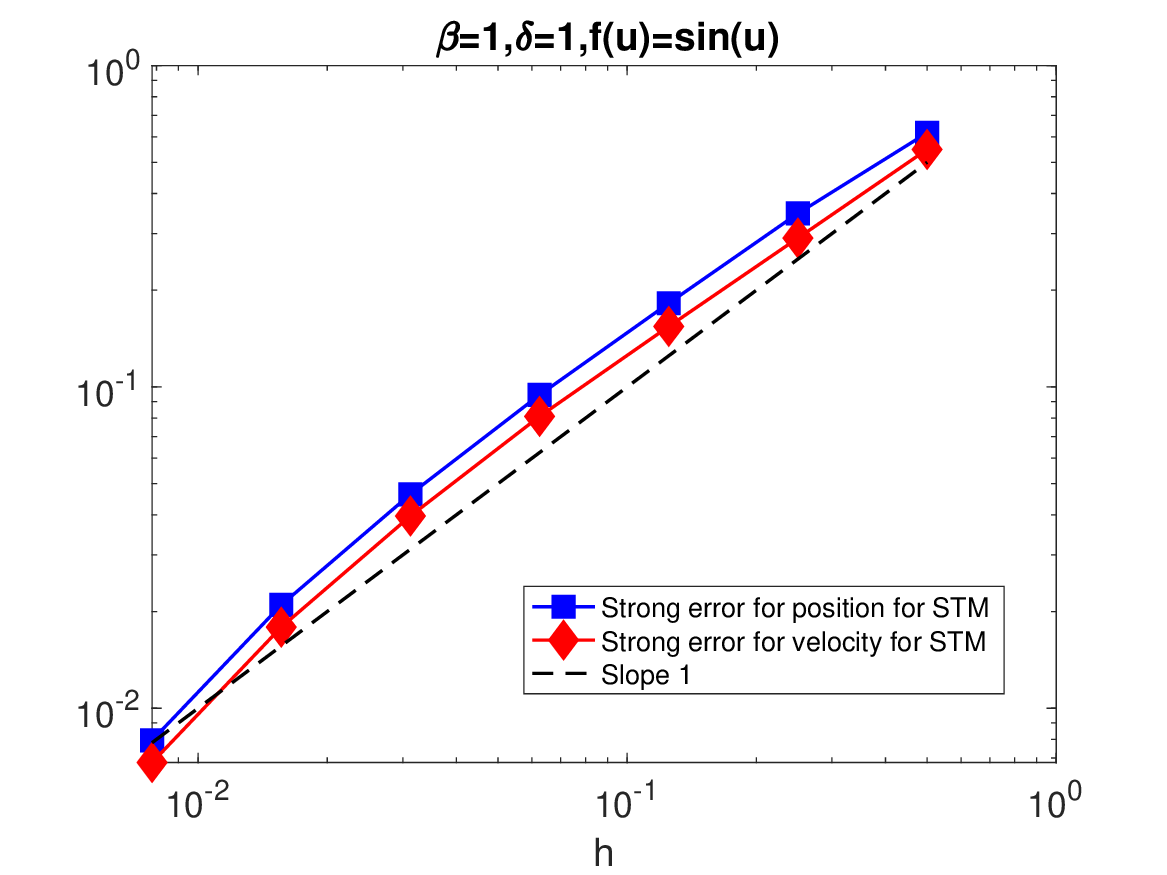}}\quad \subfigure{\includegraphics[width=0.45\textwidth]{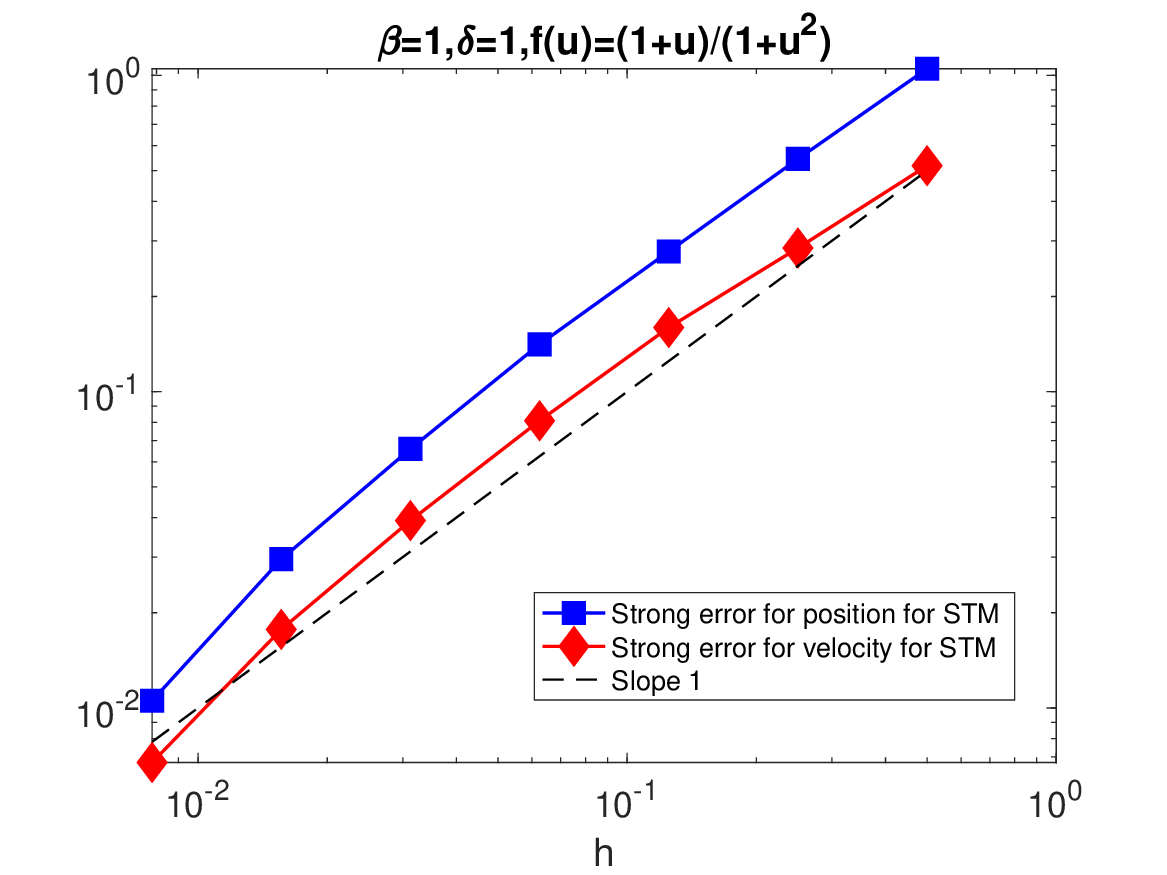}}
          \caption{Convergence in time: $L^2(\Omega; L^2(\mathbb{S}^2))$, resp. $L^2(\Omega; H^{-1}(\mathbb{S}^2))$ errors for position, resp. velocity 
          for $\beta=\delta=1$ and for the SPDE~\eqref{stoc_wave} with additive noise represented by~\eqref{eigen_q}, nonlinearities from~\eqref{testOp2} and initial valuers~\eqref{in_data1}.}
           \label{fig6}
\end{figure}

\subsection{The multiplicative case}
In this subsection, we consider the semilinear stochastic wave equation~\eqref{stoc_wave} with the nonlinearity~\eqref{op_test1} and multiplicative noise term 
$g(u)$ specified below. The problem is considered on the time interval $[0,T]$ with $T=1$ 
and initial values given by~\eqref{in_data1} with $\beta=1$. We take $Q^{1/2}\in \mathcal{L}_2(L^2(\mathbb{S}^2))$, that is the eigenvalues of the covariance operator $Q$ are given by~\eqref{eigen_q} with $\alpha=2+1e-6$.

Let us begin by illustrating the sample paths of solutions to the aforementioned SPDE with $g(u)=\sin(u)$ and $g(u)= (1+u)/(1+u^2)$. We use the stochastic trigonometric scheme~\eqref{exp_met} with time step size $h=2^{-10}$ and take the truncation index 
$\kappa=2^7$ in order to display one realization of solutions of these SPDEs in Figure~\ref{sample_mult}.

\begin{figure}
\centering
 \subfigure{\includegraphics[width=0.48\textwidth]{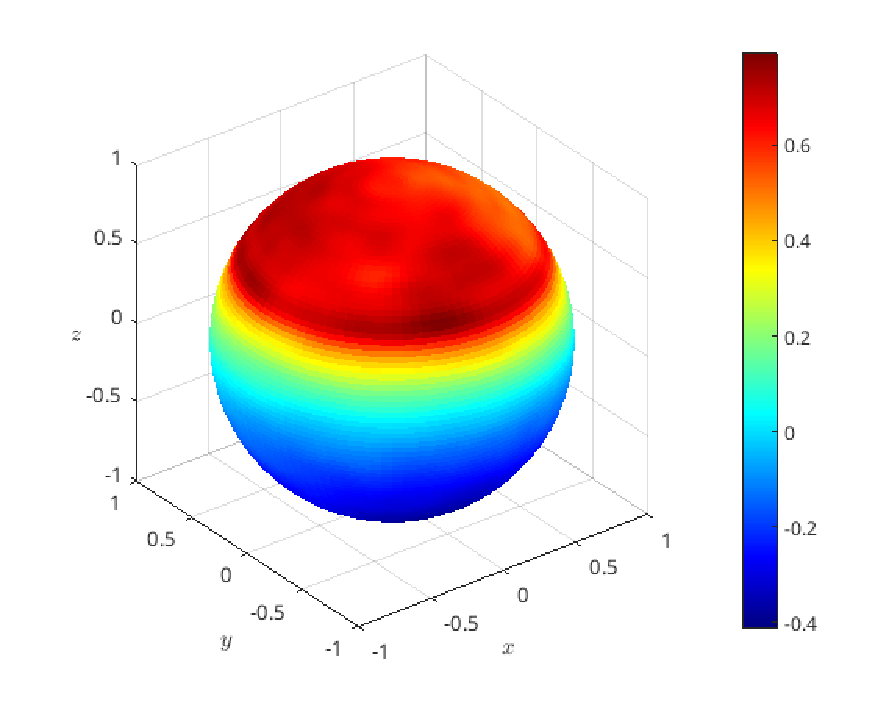}}\quad 
  \subfigure{\includegraphics[width=0.48\textwidth]{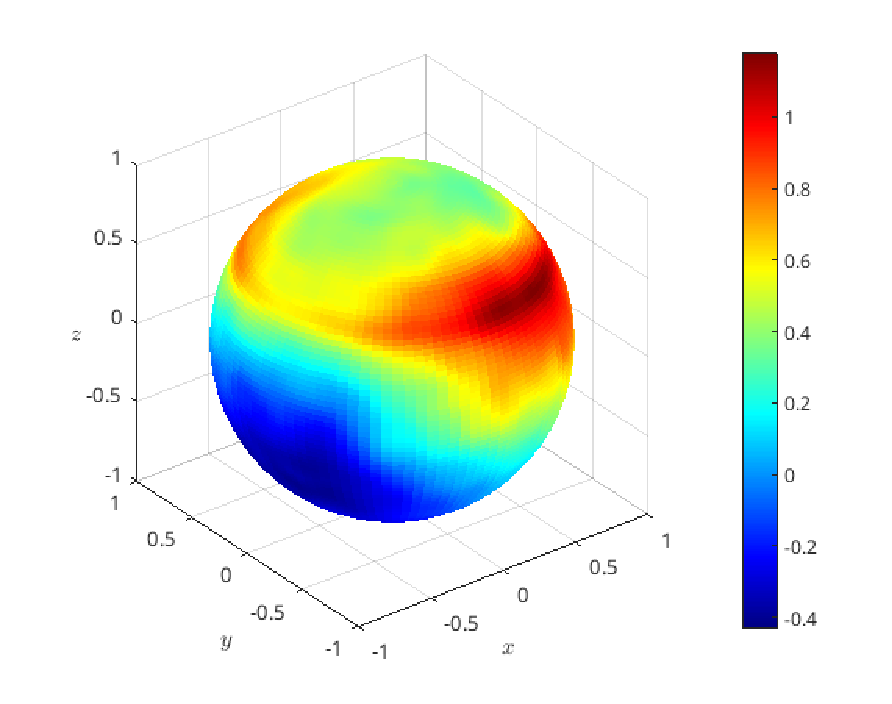}}
          \caption{Sample paths at time $1$ of the solution to the stochastic wave equation on the sphere with multiplicative noise with $g(u)=\sin(u)$ on the left and $g(u)=(1+u)/(1+u^2)$ on the right, nonlinearity \eqref{op_test1} and initial data \eqref{in_data1}, with $\beta=1$. Here, we take $\alpha=2+1e-6$. }
           \label{sample_mult}
\end{figure}
In order to investigate the temporal rate of convergence of the proposed time integrator, we take $g(u)=\sin(u)$ and thus $\delta=1$.
The reference solution is computed using the stochastic trigonometric  integrator with time step size $h_{\rm ref}=2^{-10}$ and both reference and numerical solutions are computed with a fixed truncation parameter $\kappa=2^7$. The results are presented in Figure~\ref{fig7}. In this figure, one can also observe the pathwise rates of convergence~$1$ 
in the position and velocity as proved in Theorem~\ref{tem_thm}. Finally, Figure~\ref{fig8} reports the pathwise and strong errors for the SPDE~\eqref{stoc_wave} with coefficients $f(u)=\sin(u)$ and $g(u)=(1+u)/(1+u^2)$. In this last numerical experiments, the reference solution has been computed with $h_{\rm ref}=2^{-11}$ and both reference and numerical solutions have been computed with the truncation parameter $\kappa=2^7$.

\begin{figure}
\centering
\subfigure{\includegraphics[width=0.45\textwidth]{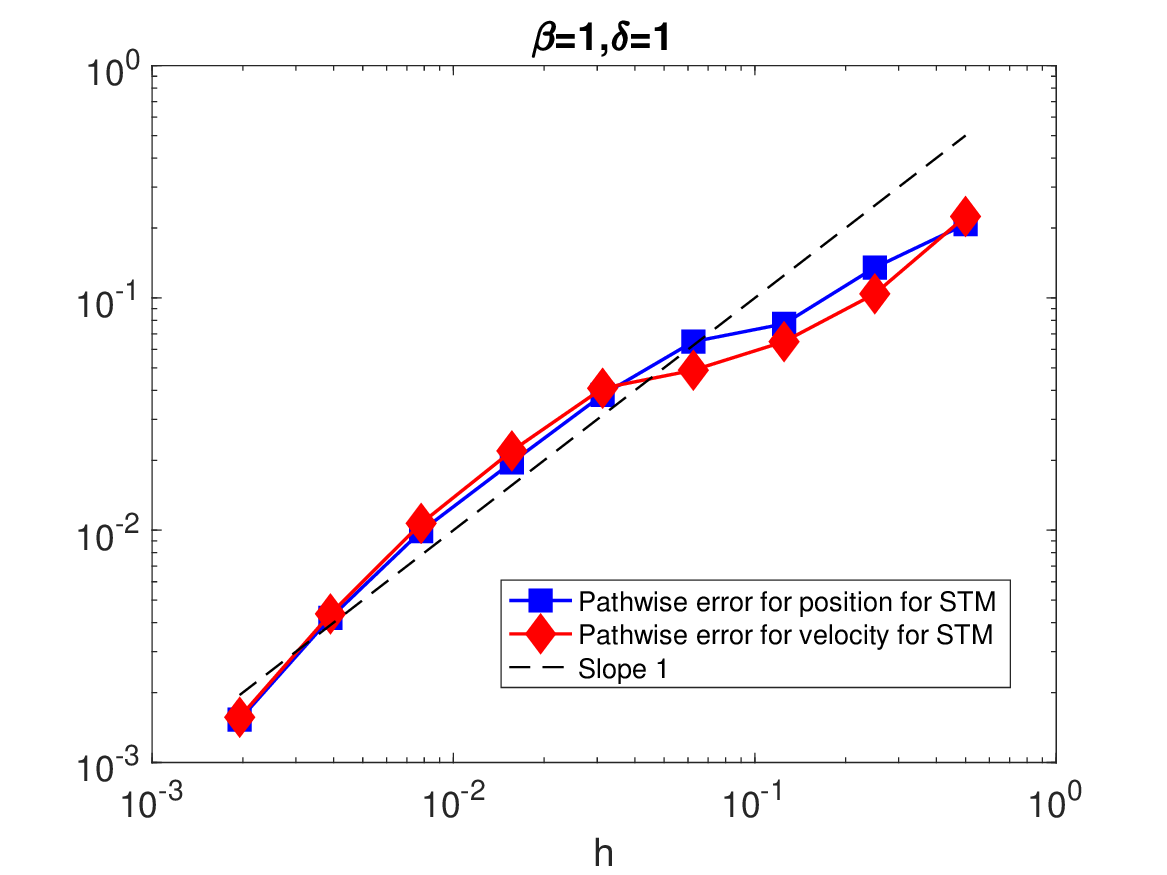}}
\quad \subfigure{\includegraphics[width=0.45\textwidth]{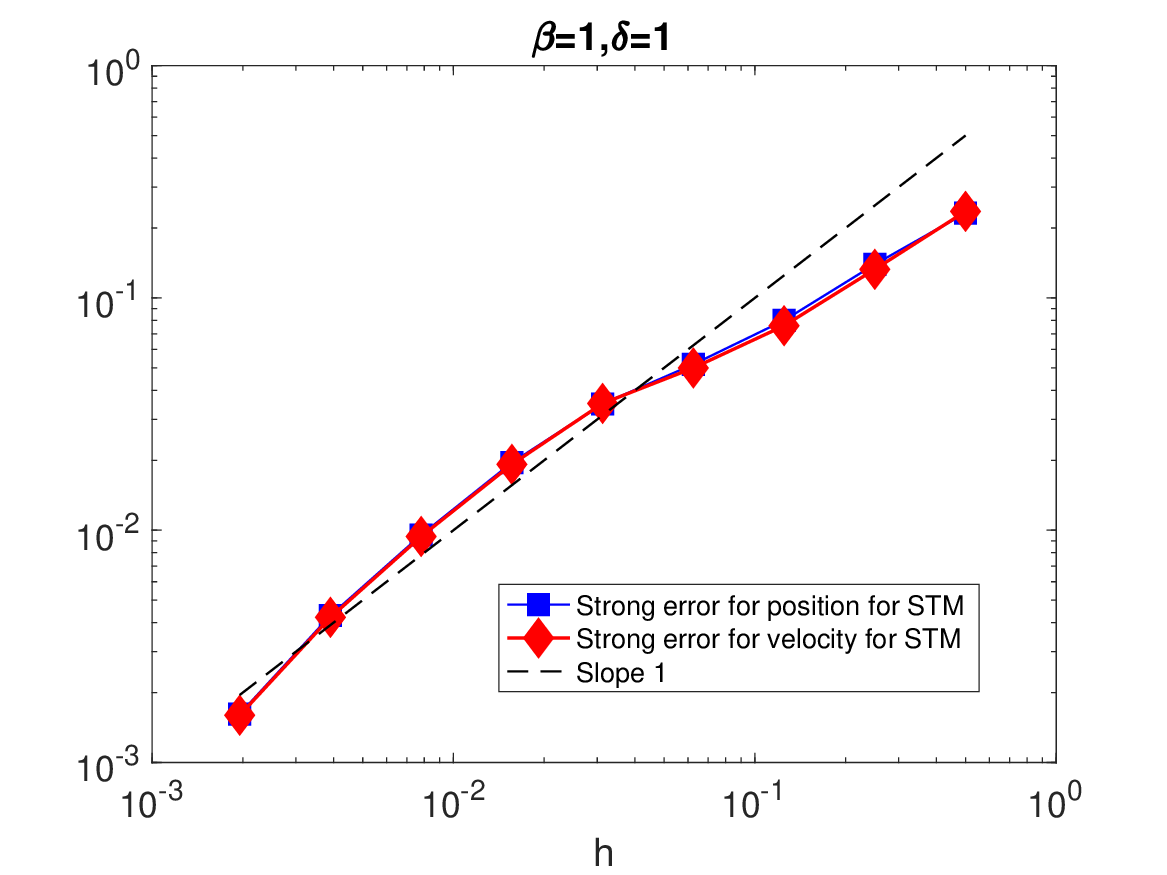}}
          \caption{Stochastic wave equation~\eqref{stoc_wave} with $f$ given in~\eqref{op_test1}, $g(u)=\sin(u)$, and initial data given in~\eqref{in_data1}: Temporal pathwise errors in $L^2(\mathbb{S}^2)$ and $H^{-1}(\mathbb{S}^2)$ for the position and velocity on the left and the corresponding strong errors on the right.}
           \label{fig7}
\end{figure}
\begin{figure}
\centering
\subfigure{\includegraphics[width=0.45\textwidth]{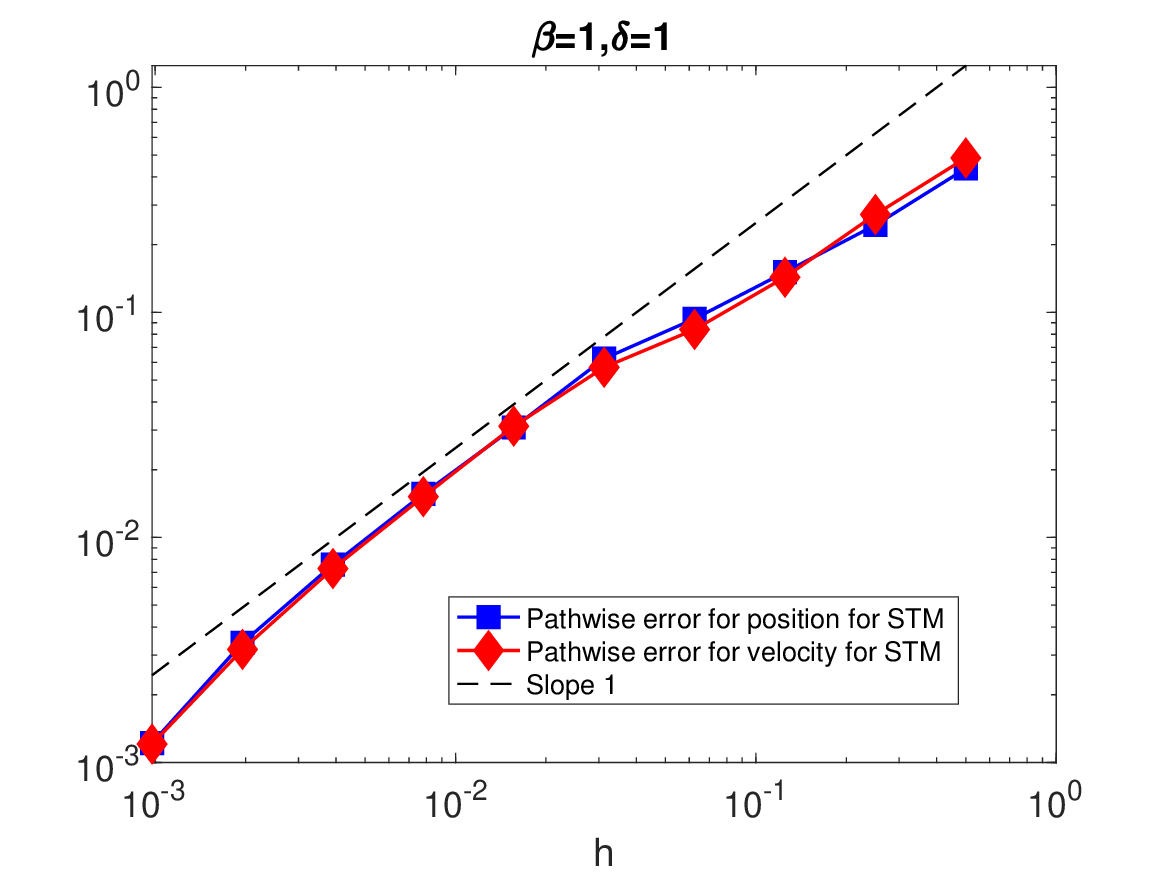}}
\quad \subfigure{\includegraphics[width=0.45\textwidth]{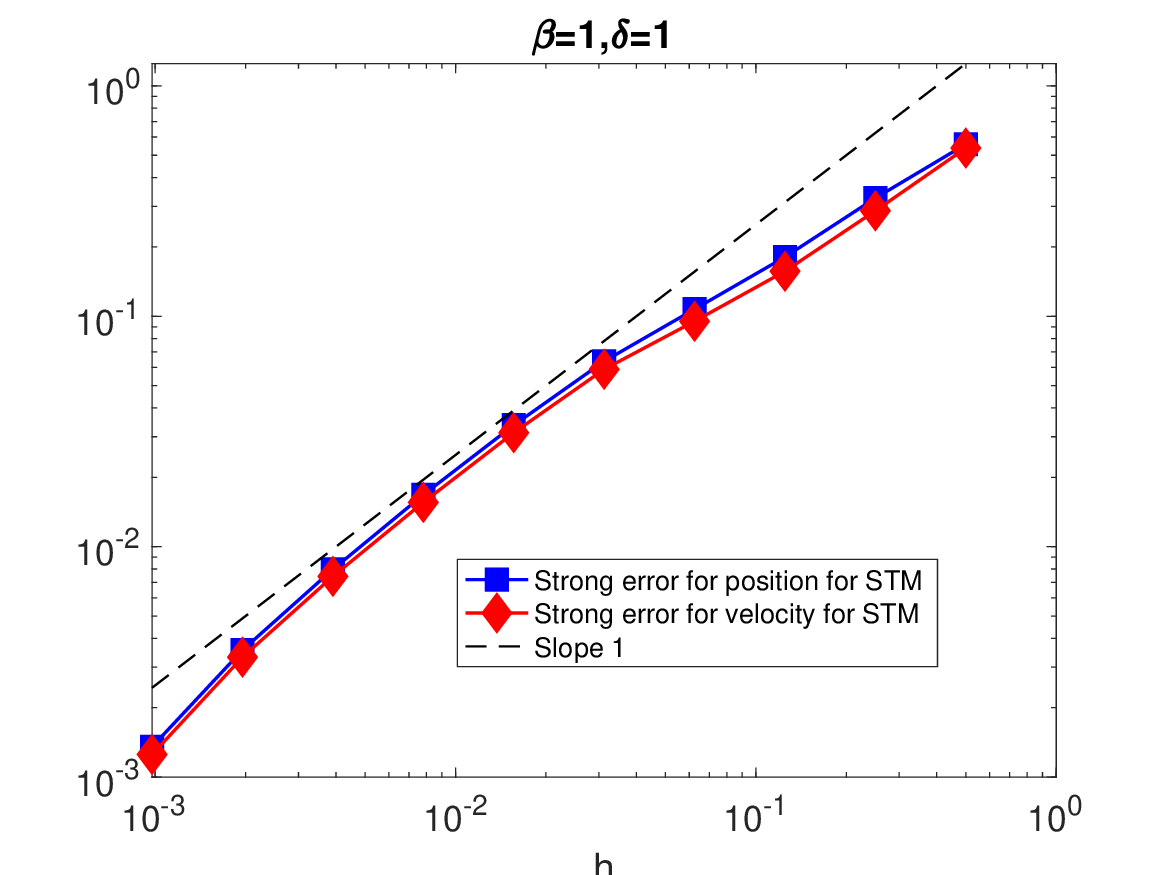}}
          \caption{Stochastic wave equation~\eqref{stoc_wave} with $f=\sin(u)$ and $g(u)=(1+u)/(1+u^2)$, and initial data given in~\eqref{in_data1}: Temporal pathwise errors in $L^2(\mathbb{S}^2)$ and $H^{-1}(\mathbb{S}^2)$ for the position and velocity on the left and the correspondent strong errors on the right.}
           \label{fig8}
\end{figure}

\bibliographystyle{abbrv}
\bibliography{labib}

\section{Acknowledgements}

SDG is a member of the INdAM Research group GNCS. SDG thanks INdAM for granting his visiting research period at Chalmers University of Technology in Gothenburg in April 2024 which allowed to start this project. His work was supported by PRIN-MUR 2022 project 20229P2HEA ``Stochastic numerical modelling for sustainable innovation'' (CUP: E53C24002280006), granted by MUR
within the scrolling of the final rankings of the PRIN 2022 call. The work of DC and AL was supported in part by the European Union (ERC, StochMan, 101088589) and by the Swedish Research Council (VR) through grants no.\ 2020-04170 and 2024-04536. The computations were performed on resources provided by the National Academic Infrastructure for Supercomputing in Sweden (NAISS) at Vera, Chalmers e-Commons at Chalmers University of Technology and partially funded by the Swedish Research Council through grant agreement no. 2022-06725.

Views and opinions expressed are however those of the author(s) only and do not necessarily reflect those of the European Union or the European Research Council Executive Agency. Neither the European Union nor the granting authority can be held responsible for them.

\end{document}